\documentclass[a4paper]{article}
\usepackage[margin=25mm]{geometry}
\usepackage{amsmath}
\usepackage{amsthm}
\usepackage{amssymb}
\usepackage{amsfonts}
\usepackage{latexsym}
\usepackage{mathtools}
\usepackage{array}
\usepackage{bm}
\usepackage{comment} 
\usepackage{authblk}

\usepackage{tikz}
\usetikzlibrary{arrows.meta, positioning, calc}

\newtheorem{prop}{Proposition}[section]
\newtheorem{lem}[prop]{Lemma}
\newtheorem{thm}[prop]{Theorem}
\newtheorem{cor}[prop]{Corollary}

\theoremstyle{definition}
\newtheorem{dfn}[prop]{Definition}

\newtheorem{que}[prop]{Question}
\newtheorem{state}[prop]{Statement}

\numberwithin{prop}{section}
\numberwithin{equation}{section}

\newenvironment{lemproof}{
  
  \begin{proof}
}{
  \end{proof}
}

\title{Frame definability in second-order arithmetic}
\author{Yuto Takeda}
\affil{Mathematical Institute, Tohoku University, Japan}

\usepackage[T1]{fontenc}

\begin{document}
\maketitle

%
%\authorrunning{Y. Takeda et al.}
%\institute{Tohoku University, Sendai, Miyagi, JPN\and
%\email{sho.shimomichi.s8@dc.tohoku.ac.jp}\and
%\email{yuto.takeda.t8@dc.tohoku.ac.jp}\and
%\email{keita.yokoyama.c2@dc.tohoku.ac.jp}}
%

%
\renewcommand{\thefootnote}{\fnsymbol{footnote}}

\begin{abstract}
We study the reverse-mathematical strength of frame definability in modal logic. The central principle is the Valuation Extension Lemma (VEL), which asserts that every assignment of propositional variables on a frame extends to a full valuation. We show that, over $\mathrm{RCA}_0$, VEL is equivalent to $\mathrm{ACA}^{+}_0$, and as are frame-definability principles for Geach axioms and for $\mathbf{GL}$. We also obtain analogous $\mathrm{ACA}^{+}_0$-equivalences for the Barcan and Converse Barcan formulas in modal predicate logic. Finally, we examine variants of VEL for $\mathbf{CTL}$ and $\mathbf{LTL}$ and locate their strengths between familiar subsystems of second-order arithmetic.

\begin{comment}
This paper investigates the logical strength of frame definability in modal logic within the framework of second-order arithmetic. We show that, over $\mathrm{RCA}_0$, frame definability for Geach axioms as well as for $\mathbf{GL}$ is equivalent to $\mathrm{ACA}^{+}_0$. For modal predicate logic, frame definability concerning the properties of variable-domain frames corresponding to the Barcan formula and the Converse Barcan formula is likewise equivalent to $\mathrm{ACA}^{+}_0$ over $\mathrm{RCA}_0$. To analyze principles related to these frame definabilities, we examine the Valuation Extension Lemma (VEL), which asserts that any assignment to propositional variables on a frame can be extended to a valuation. We prove that VEL is equivalent to $\mathrm{ACA}^{+}_0$ over $\mathrm{RCA}_0$. In particular, we show that there is a computable frame such that
every valuation on it has computational complexity of $\emptyset^{(\omega)}$. Furthermore, VEL for Computation Tree Logic (CTL) lies between $\Sigma^1_2$-$\mathrm{DC}_0$ and $\Pi^1_1$-$\mathrm{CA}_0$, whereas VEL for Linear Temporal Logic (LTL) lies between $\mathrm{ACA}^{\prime}_0$ and $\mathrm{ACA}_0$.
\end{comment}
\end{abstract}
\section{Introduction}
In this paper, we investigate frame definability for modal logic in second-order arithmetic. In particular, we analyze, from the viewpoint of reverse mathematics, frame definability for Geach axioms, which have long been studied in modal logic, and frame definability for $\mathbf{GL}$ as an example of frame definability for formulas not expressible by Geach axioms. We show that the frame-definability principles for Geach axioms and for $\mathbf{GL}$ are both equivalent to $\mathrm{ACA}^{+}_0$ over $\mathrm{RCA}_0$. We also analyze frame definability in modal predicate logic, focusing on the Barcan formula, the Converse Barcan formula, and the corresponding properties of variable domains. In particular, frame definability for both the Barcan formula and the Converse Barcan formula is equivalent to $\mathrm{ACA}^{+}_0$ over $\mathrm{RCA}_0$.
%\vspace{1mm}

Behind these results lies an important relationship between frames and valuations, or evaluation maps, on them. Given a frame and an assignment of propositional variables on that frame, we call the assertion that this assignment can be extended to a valuation the Valuation Extension Lemma, or VEL for short. In the usual setting of modal logic, VEL is treated as an almost trivial property. Indeed, in many standard texts, proofs of completeness for normal modal logics first construct an assignment of propositional variables on a frame and then use the Truth Lemma to show that the resulting extended valuation assigns truth values to formulas in the expected way; the existence of this extended valuation itself is usually not treated as a separate issue (see, for example,~\cite{MR1837791}). By contrast, in second-order arithmetic, it is natural, by analogy with Simpson's treatment of valuations for first-order logic, to define a Kripke valuation as a $0,1$-valued function on pairs consisting of a state, or possible world, in a frame and a formula. Such a valuation is a set of natural numbers, so in second-order arithmetic one must pay attention to whether it actually exists as a set. Thus, VEL is nontrivial in second-order arithmetic. 
%\vspace{1mm}

We show that VEL is equivalent to $\mathrm{ACA}^{+}_0$ over $\mathrm{RCA}_0$. Furthermore, we show that there is a computable frame such that every valuation on it has computational complexity at least that of $\emptyset^{(\omega)}$. This implies that the difficulty of VEL does not depend on how propositional variables are assigned. This further suggests that not only frame definability for particular formulas, but any frame-definability result whose proof is not too difficult, for instance one formalizable in $\mathrm{ACA}^{+}_0$, should be equivalent to $\mathrm{ACA}^{+}_0$. As further analyses of VEL, we also treat VEL for Computation Tree Logic ($\mathbf{CTL}$) and Linear Temporal Logic ($\mathbf{LTL}$). These logics are well known in computer science, and their semantics can be naturally formalized in second-order arithmetic. We show that VEL for $\mathbf{CTL}$ follows from $\Sigma^1_2$-$\mathrm{DC}_0$ and implies $\Pi^1_1$-$\mathrm{CA}_0$. We also show that VEL for $\mathbf{LTL}$ follows from $\mathrm{ACA}^{\prime}_0$ and implies $\mathrm{ACA}_0$. These results provide a foundation for the reverse-mathematical study of $\mathbf{CTL}$ and $\mathbf{LTL}$, and form a first step toward analyzing a unified logic such as $\mathbf{CTL}^{*}$.
%\vspace{1mm}

Several studies have analyzed logic within second-order arithmetic. For example, over $\mathrm{RCA}_0$, Simpson~\cite{MR2517689} showed that G\"odel's completeness theorem is equivalent to $\mathrm{WKL}_0$, and Yamazaki~\cite{MR2010178} showed that the strong completeness theorem for intuitionistic predicate logic is equivalent to $\mathrm{ACA}_0$. For modal logic, we previously showed that the weak completeness theorem for propositional modal logic is provable in $\mathrm{RCA}_0$. We then showed that proving the strong completeness theorem via canonical models requires $\mathrm{ACA}_0$, whereas the strong completeness theorem itself is equivalent to $\mathrm{WKL}_0$ over $\mathrm{RCA}_0$~\cite{10.1007/978-3-031-95908-0_32}. Moreover, we analyzed two fundamental theorems in modal logic: the Hennessy--Milner theorem and the van Benthem characterization theorem. In~\cite{takeda2026bisimulationssecondorderarithmetic}, we show that, over $\mathrm{RCA}_0$, the Hennessy--Milner theorem is equivalent to $\mathrm{ACA}_0$. For the van Benthem characterization theorem, we showed that the semantic form is provable in $\mathrm{RCA}_0$, the syntactic form is provable in $\mathrm{PRA}$, and the hybrid form is equivalent to the weak completeness theorem for first-order logic over $\mathrm{RCA}_0$. The present work contains both foundational and further developments in the reverse-mathematical study of modal logic.
%\vspace{1mm}

In Section 2, we study frame definability for modal propositional logic. In Subsection 2.1, we show that VEL is equivalent to $\mathrm{ACA}^{+}_0$ over $\mathrm{RCA}_0$. We also show that Restricted VEL is equivalent to $\mathrm{ACA}^{\prime}_0$ over $\mathrm{RCA}_0$. Furthermore, we show that there is a computable frame such that every valuation on it has computational complexity at least that of $\emptyset^{(\omega)}$. In Subsection 2.2, we discuss the relationship between frame definability and VEL. In particular, we show that frame definability for Geach axioms and for $\mathbf{GL}$ is equivalent to $\mathrm{ACA}^{+}_0$ over $\mathrm{RCA}_0$. We also show that VEL for image-finite frames is equivalent to $\mathrm{ACA}_0$ over $\mathrm{RCA}_0$. In Subsection 2.3, we treat frame definability for modal predicate logic. In particular, we show that frame definability for the Barcan formula and for the Converse Barcan formula is equivalent to $\mathrm{ACA}^{+}_0$ over $\mathrm{RCA}_0$. In Section 3, we further investigate VEL. In Subsection 3.1, we show that VEL for $\mathbf{CTL}$ follows from $\Sigma^1_2$-$\mathrm{DC}_0$ and that, over $\mathrm{RCA}_0$, it implies $\Pi^1_1$-$\mathrm{CA}_0$. We also show that VEL for image-finite frames is equivalent to $\mathrm{ACA}^{+}_0$ over $\mathrm{RCA}_0$. In Subsection 3.2, we show that VEL for $\mathbf{LTL}$ follows from $\mathrm{ACA}^{\prime}_0$ and that, over $\mathrm{RCA}_0$, it implies $\mathrm{ACA}_0$. For the formalization of modal logic over $\mathrm{RCA}_0$, we refer to~\cite{10.1007/978-3-031-95908-0_32}.
%\vspace{1mm}

%\subsection{Framework}
We work in the framework of subsystems of second-order arithmetic, with base theory $\mathrm{RCA}_0$. $\mathrm{RCA}_0$ consists of the axioms of Robinson arithmetic, together with the $\Sigma^0_1$-induction scheme and the $\Delta^0_1$-comprehension scheme. $\mathrm{WKL}_0$ consists of the axioms of $\mathrm{RCA}_0$ together with a formalization of weak K\"onig's lemma. $\mathrm{ACA}_0$ consists of the axioms of $\mathrm{RCA}_0$ with the comprehension scheme for all arithmetical formulas. $\Pi^1_1$-$\mathrm{CA}_0$ consists of the axioms of $\mathrm{RCA}_0$ with the comprehension scheme for all $\Pi^1_1$ formulas. Furthermore, we consider two variants of $\mathrm{ACA}_{0}$, namely, $\mathrm{ACA}^{\prime}_{0}$ and $\mathrm{ACA}^{+}_{0}$. The system $\mathrm{ACA}^{\prime}_{0}$ consists of $\mathrm{RCA}_{0}$ plus the axiom stating closure under all finite jumps, that is, the assertion that for every $X$ and every $n$, there is a set $Y$ such that $Y=X^{(n)}$; the system $\mathrm{ACA}^{+}_{0}$ consists of $\mathrm{RCA}_{0}$ plus the axiom
stating the existence of the $\omega$-jump of any set, that is, the assertion that for every $X$, there is a set $Y$ such that $Y=X^{(\omega)}$. For other principles, including induction and choice schemes, see~\cite{MR2517689}.
\section{Frame definability in second-order arithmetic}
In this section, we consider frame definability over $\mathrm{RCA}_0$.  As usual, we identify finite sequences of symbols of the language for modal logic with natural numbers. Let $\omega=\{0, 1, \dots\}$ be the set of standard natural numbers and let $\mathbb{N}$ be the first-order part of the model of $\mathrm{RCA}_0$ under consideration. For the formalization of modal logic over $\mathrm{RCA}_0$, we refer to~\cite{10.1007/978-3-031-95908-0_32}.

Let $\mathrm{Prop}=\{p,q,r,\dots\}$ be an infinite set of propositional variables. In $\mathrm{RCA}_0$, the set of all atomic formulas, $\mathrm{Atm}=\mathrm{Prop}\cup\{\bot\}$, and the set $\mathrm{Fml}$ of modal formulas built from a single modal operator $\Box$ is represented as a set of natural numbers. We regard $\Diamond\varphi$ as an abbreviation for $\neg\Box\neg\varphi$.

\begin{dfn}[$\mathrm{RCA}_0$~\cite{10.1007/978-3-031-95908-0_32}]
Let $p_0, p_1, \dots$ be an enumeration of $\mathrm{Atm}$.
We define the \emph{degree function} $\mathrm{deg} : \mathrm{Fml}\rightarrow\mathbb{N}$ by primitive recursion:
\begin{itemize}
    \item $\mathrm{deg}(p_n)=0$ for all $n$,
    \item $\mathrm{deg}(\bot)=0$,
    \item $\mathrm{deg}(\varphi\rightarrow\psi)=\max\{\mathrm{deg}(\varphi), \mathrm{deg}(\psi)\}$,
    \item $\mathrm{deg}(\Box\varphi)=1+\mathrm{deg}(\varphi)$.
\end{itemize}
\end{dfn}

\begin{dfn}[$\mathrm{RCA}_0$]
Let $n\in\mathbb{N}$ and $\varphi\in\mathrm{Fml}$. Then we define $\Box^{n}\varphi$ by the following:
\begin{itemize}
    \item $\Box^{0}\varphi\equiv\varphi$,
    \item $\Box^{n+1}\varphi\equiv\Box(\Box^{n}\varphi)$.
\end{itemize}
Similarly, we define $\Diamond^{n}\varphi$. 
\end{dfn}

We formalize the provability predicate for normal modal logic.

\begin{dfn}[$\mathrm{RCA}_0$~\cite{10.1007/978-3-031-95908-0_32}]
Let $\Gamma\subset\mathrm{Fml}$. We define the following predicates:
\begin{equation*}
\begin{split}
\mathrm{Prf}_{\mathbf{K}}(\Gamma, p)\equiv p&\in\mathrm{Seq}\land\forall k(k<lh(p)\rightarrow p(k)\in\mathrm{Fml})\\&\land\forall k\Bigl(k<lh(p)\rightarrow\Bigl(p(k)\in\Gamma\lor p(k)\in\mathrm{Axm}\\&\lor(\exists i<k\exists j<k)(p(i)=p(j)\rightarrow p(k))\lor(\exists i<k)(p(k)=\Box p(i))\\&\lor(\exists i<k)(\exists\sigma:\{q<p(i)\mid q\in\mathrm{Prop}\}\rightarrow\{\varphi<p(k)\mid\varphi\in\mathrm{Fml}\})\\&(p(k)=\overline{\sigma}(p(i)))\Bigr)\Bigr),
\end{split}
\end{equation*}
\begin{equation*}
\begin{split}
\mathrm{Pbl}_{\mathbf{K}}(\Gamma, \varphi)\equiv&\exists\psi_1,\dots,\exists\psi_n\in\Gamma\exists p\\&\bigl(\mathrm{Prf}_{\mathbf{K}}(\emptyset, p)\land(\exists i<lh(p))(p(i)=(\psi_1\land\dots\land\psi_n\rightarrow\varphi))\bigr), 
\end{split}
\end{equation*}
\end{dfn}
where $\overline{\sigma}$ is the uniform substitution for a finite function $\sigma:\subseteq\mathrm{Prop}\rightarrow\mathrm{Fml}$ (see~\cite{10.1007/978-3-031-95908-0_32}). 

\subsection{The Valuation Extension Lemma}

We next formalize the Kripke semantics.
\begin{dfn}[$\mathrm{RCA}_0$~\cite{10.1007/978-3-031-95908-0_32}]
A \emph{Kripke model} is a tuple $M=(W, R, V)$ satisfying the following conditions:
\begin{enumerate}
    \item $W\subseteq\mathbb{N}$ is a non-empty set,
    \item $R$ is a binary relation on $W$, i.e., $R\subseteq W\times W$,
    \item %$V$ is a function which assigns a truth value to each pair of a propositional formula and an element of $W$, i.e., 
$V : W\times\mathrm{Fml} \rightarrow\{0, 1\}$,
    \item $V(w, \bot)=0$ for any $w\in W$, 
    \item $V(w, \varphi\rightarrow\psi)=1-V(w, \varphi)(1-V(w, \psi))$ for any $w\in W$ and any $\varphi, \psi\in\mathrm{Fml}$, 
    \item $V(w, \Box\varphi)=1\iff\forall v\in W(wRv\rightarrow V(v, \varphi)=1)$ for any $w\in W$ and any $\varphi\in\mathrm{Fml}$.
\end{enumerate}
A pair $(W, R)$ is called a \emph{frame}, and a function $V$ is called a \emph{valuation} on $(W, R)$. A frame $F=(W, R)$ is \emph{image-finite} if for all $w\in W$, $wR=\{v\in W\mid wRv\}$ is a finite set.
\vspace{1mm}

For each $n$, let $\mathrm{Fml}[n]$ be the set of all modal formulas $\varphi$ such that $\mathrm{deg}(\varphi)\leq n$. We say that a valuation $V$ is \emph{$n$-restricted} if its domain is $W \times \mathrm{Fml}[n]$.
\end{dfn}
\begin{comment}
In general, given a frame $F=(W, R)$ and a function $v: W\times\mathrm{Prop}\rightarrow2$, a valuation $V$ on $F$ that extends $v$ is unique, if it exists. This fact can be shown by induction on the complexity of formulas.

\begin{prop}[$\mathrm{RCA}_0$]
Let $F=(W, R)$ be a frame and $v:W\times\mathrm{Prop}\rightarrow2$. Then the valuation $V$ on $F$ extending $v$ is unique; that is, the following holds: Let $V_1$ and $V_2$ be valuations. If $V_1$ and $V_2$ are extensions of $v$ respectively, then $V_1(w, \varphi)=V_2(w, \varphi)$ for all $w\in W$ and $\varphi\in\mathrm{Fml}$.
\end{prop}

\begin{proof}
This follows by $\Sigma^0_1$-induction.
\end{proof}
\end{comment}

\begin{comment}
\begin{dfn}[$\mathrm{RCA}_0$]
Let $F=(W, R)$ be a frame.
\begin{itemize}
    \item $F$ is \emph{appropriate to $T$} if $R$ is reflexive,
    \item $F$ is \emph{appropriate to $B$} if $R$ is symmetric,
    \item $F$ is \emph{appropriate to $4$} if $R$ is transitive,
    \item $F$ is \emph{appropriate to $5$} if $R$ is Euclidean,
    \item $F$ is \emph{appropriate to $D$} if $R$ is serial,
    \item $F$ is \emph{appropriate to $.2$} if $R$ is directed,
    \item $F$ is \emph{appropriate to $L$} if $R$ is transitive and $W$ has no infinite ascending sequences by $R$.
\end{itemize}

In general, $F$ is \emph{appropriate to $\mathrm{\Sigma}\subseteq\{T, B, 4, 5, D, .2, L\}$} if for all $\varphi\in\mathrm{\Sigma}$, $F$ is appropriate to $\varphi$.
\end{dfn}
\end{comment}

We also use the following notation.
Let $F$ be a frame, let $M$ be a model, let $w\in W$, and let $\varphi\in\mathrm{Fml}$. We define
\begin{align*}
&M, w\Vdash\varphi\equiv V(w, \varphi)=1, \quad \quad M\Vdash\varphi\equiv(\forall w\in W)(V(w, \varphi)=1),\\
&F\Vdash\varphi\equiv(\forall V\text{: valuation on $F$})(\forall w\in W)(V(w, \varphi)=1).
\end{align*}

Given a frame $F$, a set of formulas $\mathrm{\Gamma}$, and $\varphi\in\mathrm{Fml}$, we let
\begin{align*}
F\Vdash\mathrm{\Gamma}&\equiv(\forall\psi\in\mathrm{\Gamma})(F\Vdash\psi),\\
\mathrm{\Gamma}\Vdash\varphi&\equiv(\forall F\text{: frame and }\forall V\text{: valuation on }F)((F, V)\Vdash\mathrm{\Gamma}\rightarrow (F, V)\Vdash\varphi).
\end{align*}

In the above definition, the valuation $V$ should assign truth values to all formulas. Within $\mathrm{RCA}_{0}$, one cannot extend an assignment for atomic formulas to a full valuation.

\begin{thm}[The Valuation Extension Lemma, $\mathrm{ACA}^{+}_{0}$]\label{pextend}
For each frame $F=(W, R)$ and $v:W\times\mathrm{Prop}\rightarrow2$, there exists a valuation $V$ on $(W, R)$ such that $V$ is an extension of $v$.
\end{thm}

\begin{proof}
We first give a proof from the perspective of computability theory. Afterward, we verify that these arguments can be formalized within $\mathrm{ACA}^{+}_0$.
\vspace{1mm}

\begin{lemproof}[\textbf{Proof via computability theory}:]
Let $F=(W, R)$ be a frame and $v: W\times\mathrm{Prop}\rightarrow2$. We construct a valuation $V: W \times \mathrm{Fml} \to 2$ extending $v$ in a bottom-up manner, using $(F\oplus v)^{(\omega)}$ as an oracle.
\vspace{1mm}

Step $0$: We define $V_{0}:W\times\mathrm{Fml}[0]\rightarrow2$ as follows: for all $w\in W$ and $p\in\mathrm{Prop}$, $V_0(w, p)=v(w, p)$, and for all $w\in W$ and $\varphi, \psi\in\mathrm{Fml}[0]$ for which the values of $V_{0}$ have already been defined, $V_{0}(w, \varphi\rightarrow\psi)=1-V_0(w, \varphi)(1-V_0(w, \psi))$. Then, we search for an index $e_{0}$ of $V_{0}$ using $F\oplus v$. From this point forward, we only need to specify the clauses for $\Box$.
\vspace{1mm}

Step $k$: We assume that $V_{k-1}: W \times \mathrm{Fml}[k-1] \to 2$ is defined and $e_{k-1}$ is an index of $V_{k-1}$. Fix $w\in W$ and $\varphi\in\mathrm{Fml}[k-1]$. Then, using $(F \oplus v)^{(k)}$ as an oracle, we determine whether $\forall u \in W (wRu \rightarrow V_{k-1}(u, \varphi))$ holds. If $\forall u \in W (wRu \rightarrow V_{k-1}(u, \varphi))$ holds, then $V_{k}(w, \Box\varphi)=1$; otherwise, $V_{k}(w, \Box\varphi)=0$. Using $(F \oplus v)^{(k)}$ as an oracle, we compute an index $e_{k}$ for the $k$-restricted valuation $V_{k} : W \times \mathrm{Fml}[k] \to 2$ defined in this manner.
\vspace{1mm}

In this construction, we obtain a sequence of indices $E=\{e_{k}\}_{k}$. Thus, using this $E$, we define the valuation $V$ such that $V(w, \varphi) = V_{k}(w, \varphi)$ for all $w \in W$ and $\varphi \in \mathrm{Fml}$ with $\mathrm{deg}(\varphi) = k$. 
\end{lemproof}
\vspace{1mm}

We now verify that the above argument can be formalized in $\mathrm{ACA}^{+}_0$. In $\mathrm{ACA}_{0}$, we can define certain concepts from computability theory involving Turing jumps (see Simpson \cite{MR2517689} for details).
Let $X, Y\subseteq\mathbb{N}$. Then, $\mathrm{TJ}(X)$ denotes the Turing jump of $X$ as in~\cite{MR2517689}. If $Y\leq_{T}\mathrm{TJ}(X)$, we can consider a $\Delta^{0, X}_2$-index of $Y$ (obtained as a $\mathrm{TJ}(X)$-recursive index of $Y$). Moreover, the statement ``$e$ is a $\mathrm{TJ}(X)$-recursive index of a set'' is arithmetical, and in such cases, we write $\Delta^{0, X}_2[e]$ for the set indexed by $e$. 
%\vspace{1mm}

By $\mathrm{ACA}^{+}_0$, we have $(F\oplus v)^{(\omega)}$. At each step $k$, we need to find a $\Delta^{0, (F\oplus v)^{(k-1)}}_2$-index for the $k$-restricted valuation $V_{k}$. This is possible by using $(F\oplus v)^{(k-1)}$ as an oracle (that is, by using $(F\oplus v)^{(\omega)}$). In particular, these arguments can be carried out within $\mathrm{ACA}_0$. Furthermore, using arithmetical comprehension (and arithmetical induction), we obtain a sequence of indices $E$. This completes the proof.
\end{proof}

\begin{cor}[The Restricted Valuation Extension Lemma, $\mathrm{ACA}^{\prime}_{0}$]\label{pextend2}
For each frame $F=(W, R)$, $v:W\times\mathrm{Prop}\rightarrow2$, and $k\in\mathbb{N}$, there exists a $k$-restricted valuation $V$ on $(W, R)$ such that $V$ is an extension of $v$.
\end{cor}

\begin{proof}
By Theorem~\ref{pextend}, it is sufficient to use $(F\oplus v)^{k}$ as an oracle to construct a $k$-restricted valuation.
\end{proof}

On the other hand, if $k$ is standard, a $k$-restricted valuation can be constructed within $\mathrm{ACA}_0$.

\begin{cor}[$\mathrm{ACA}_{0}$]\label{pxtendstandard}
For each frame $F=(W, R)$, $v:W\times\mathrm{Prop}\rightarrow2$, and $k\in\omega$, there exists a $k$-restricted valuation $V$ on $(W, R)$ such that $V$ is an extension of $v$.
\end{cor}

\begin{proof}
It is enough to consider the case where $k=2$. By Theorem~\ref{pextend}, a $2$-restricted valuation can be constructed relative to $(F \oplus v)^{\prime\prime}$. Furthermore, the existence of $(F \oplus v)^{\prime\prime}$ is guaranteed by arithmetical comprehension.
\end{proof}

Next, we consider the logical strength of VEL. In fact, over $\mathrm{RCA}_0$, VEL implies $\mathrm{ACA}_0$. To show this, we use the following proposition.

\begin{prop}[$\mathrm{RCA}_0$~\cite{MR2517689}]\label{acaeqijf}
$\mathrm{ACA}_0$ is equivalent to the statement that for each injective function $f:\mathbb{N}\rightarrow\mathbb{N}$, there exists a set $A$ such that $\forall x(x\in A\leftrightarrow\exists y(f(y)=x))$.
\end{prop}

\begin{prop}[$\mathrm{RCA}_0$]\label{velraca}
VEL implies $\mathrm{ACA}_0$.
\end{prop}

\begin{proof}
By Proposition~\ref{acaeqijf}, it suffices to show that for an injective function $f$, there exists a range set of $f$. We take an injective function $f:\mathbb{N}\rightarrow\mathbb{N}$. Let $\mathrm{Prop}=\{p_{i}\}_{i\in\mathbb{N}}$, $W=\mathbb{N}$ and $R=W\times W$. Then we define $v:W\times\mathrm{Prop}\rightarrow2$ by the following: for all $p_{i}\in\mathrm{Prop}$ and $w\in W$, $v(w, p_{i})=1\Leftrightarrow f(w)=i$. By assumption, there exists the full valuation $V: W\times\mathrm{Fml}\rightarrow2$ on $(W, R)$ such that $V$ is an extension of $v$. Then we define $I=\{w\in W\mid V(w, \Diamond p_{w})=1\}$. $I$ is the range of $f$. In fact, for all $n$, $\exists m(f(m)=n)\Leftrightarrow\exists m(v(m, p_{n})=1)\Leftrightarrow V(n, \Diamond p_{n})=1\Leftrightarrow n\in I$.
\end{proof}

In fact, regardless of the assignment of truth values to propositional variables, the complexity of a valuation on a frame may increase solely because of the frame. 

\begin{prop}[$\mathrm{RCA}_0$]\label{wwwvelaca}
The following statement implies $\mathrm{ACA}_0$: For each frame $F=(W, R)$, there exists a $1$-restricted valuation $V$ on $(W, R)$.
\end{prop}

\begin{proof}
By Proposition~\ref{acaeqijf}, it suffices to show that for an injective function $f$, there exists a range set of $f$. We take an injective function $f:\mathbb{N}\rightarrow\mathbb{N}$. Let $W=\mathbb{N}$ and $R=\{\langle{n, m\rangle}\mid f(m)=n\}$. By assumption, there exists a  $1$-restricted valuation $V: W\times\mathrm{Fml}[1]\rightarrow2$. Then we define $I=\{w\in W\mid V(w, \Diamond\top)=1\}$. $I$ is the range of $f$. In fact, for all $n$, $\exists m(f(m)=n)\Leftrightarrow\exists m(nRm)\Leftrightarrow V(n, \Diamond\top)=1\Leftrightarrow n\in I$. 
\end{proof}

Based on these considerations, we investigate the relationship between a frame and a valuation on it from a computability perspective. We say that a frame $F = (W, R)$ is computable if both $W$ and $R \subseteq W \times W$ are computable.

\begin{thm}\label{comframepr}
Let $A$ be a set and $k\in\omega$. Then there exists a $A$-computable frame $F=(W, R)$ such that for any valuation $V:W\times \mathrm{Fml}\rightarrow 2$ on $(W, R)$, $V$ computes $A^{(2k+1)}$.
\end{thm}

\begin{proof}
Let $A$ be a set and $k\in\omega$. Let $\{\exists n_{1}\forall n_{2}\exists n_{3}\dots\forall n_{2k}\exists n_{2k+1}\theta_{i}(\vec{n}, x, X)\}_{i\in\omega}$ be a computable listing of $\Sigma^{0}_{2k+1}$-formulas with no free variables except a number variable $x$ and a set variable $X$. (The $\theta_{i}$'s are bounded formulas).  Let $W=\omega^{<2k+4}$. Then we define $R\subseteq W\times W$ as follows: for $\sigma, \tau\in W$, 
\begin{align*}
\sigma R\tau\Leftrightarrow\sigma\preceq\tau&\land(lh(\tau)=lh(\sigma)+1)\\
&\land(lh(\sigma)=2k+2\rightarrow\theta_{\sigma_{1}}(\sigma_{2}, \sigma_{3}, \dots, \sigma_{lh(\sigma)-1}, \tau_{lh(\tau)-1}, \sigma_{0}, A)).
\end{align*} 

This gives the frame $(W, R)$. The following diagram illustrates the part of $(W, R)$ reachable from $\langle{n, e\rangle}$ in the case $k=2$. The transition relation $R$ extends from each node to the nodes immediately above it. In this example, $\theta_{e}(n_{1}, n_{2}, \dots, n_{5}, n, A)$ holds.
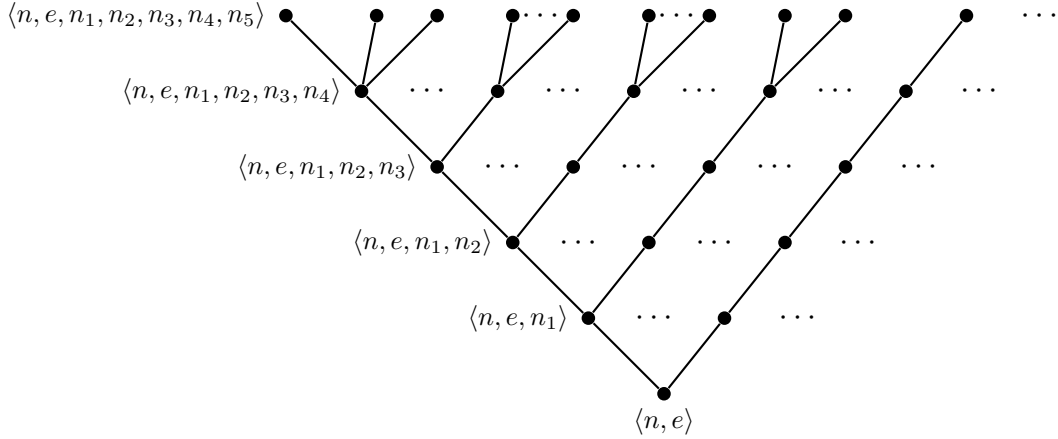
\begin{center}
\begin{tikzpicture}[
    scale=1,
    node/.style={circle, fill=black, inner sep=1.8pt}, % Slightly enlarge nodes (1.2pt -> 1.8pt)
    label node/.style={font=\normalsize} % Changed from \small to \large
]

% 1. Draw Main Spine Nodes and Labels
\node[node] (n0) at (0, 0) {};
\node[label node, below=2pt] at (n0) {$\langle n, e \rangle$};

\node[node] (n1) at (-1, 1) {};
\node[label node, left=4pt] at (n1) {$\langle n, e, n_{1} \rangle$};

\node[node] (n2) at (-2, 2) {};
\node[label node, left=4pt] at (n2) {$\langle n, e, n_{1}, n_{2} \rangle$};

\node[node] (n3) at (-3, 3) {};
\node[label node, left=4pt] at (n3) {$\langle n, e, n_{1}, n_{2}, n_{3} \rangle$};

\node[node] (n4) at (-4, 4) {};
\node[label node, left=4pt] at (n4) {$\langle n, e, n_{1}, n_{2}, n_{3}, n_{4} \rangle$};

\node[node] (n5) at (-5, 5) {};
\node[label node, left=4pt] at (n5) {$\langle n, e, n_{1}, n_{2}, n_{3}, n_{4}, n_{5} \rangle$};

% Main spine edges
\draw[thick] (n0) -- (n1);
\draw[thick] (n1) -- (n2);
\draw[thick] (n2) -- (n3);
\draw[thick] (n3) -- (n4);
\draw[thick] (n4) -- (n5);

% 2. Draw Right Spines
% Spine from L0
\node[node] (s0_1) at (0.8, 1) {};
\node[node] (s0_2) at (1.6, 2) {};
\node[node] (s0_3) at (2.4, 3) {};
\node[node] (s0_4) at (3.2, 4) {};
\draw[thick] (n0) -- (s0_1) -- (s0_2) -- (s0_3) -- (s0_4);

% Spine from L1
\node[node] (s1_2) at (-0.2, 2) {};
\node[node] (s1_3) at (0.6, 3) {};
\node[node] (s1_4) at (1.4, 4) {};
\draw[thick] (n1) -- (s1_2) -- (s1_3) -- (s1_4);

% Spine from L2
\node[node] (s2_3) at (-1.2, 3) {};
\node[node] (s2_4) at (-0.4, 4) {};
\draw[thick] (n2) -- (s2_3) -- (s2_4);

% Spine from L3
\node[node] (s3_4) at (-2.2, 4) {};
\draw[thick] (n3) -- (s3_4);

% 3. Draw Level 5 Forks
% Fork from main L4 node (-4, 4)
\node[node] (f_m_1) at (-3.8, 5) {};
\node[node] (f_m_2) at (-3.0, 5) {};
\draw[thick] (n4) -- (f_m_1);
\draw[thick] (n4) -- (f_m_2);

% Fork from s3_4 (-2.2, 4)
\node[node] (f_3_1) at (-2.0, 5) {};
\node[node] (f_3_2) at (-1.2, 5) {};
\draw[thick] (s3_4) -- (f_3_1);
\draw[thick] (s3_4) -- (f_3_2);

% Fork from s2_4 (-0.4, 4)
\node[node] (f_2_1) at (-0.2, 5) {};
\node[node] (f_2_2) at (0.6, 5) {};
\draw[thick] (s2_4) -- (f_2_1);
\draw[thick] (s2_4) -- (f_2_2);

% Fork from s1_4 (1.4, 4)
\node[node] (f_1_1) at (1.6, 5) {};
\node[node] (f_1_2) at (2.4, 5) {};
\draw[thick] (s1_4) -- (f_1_1);
\draw[thick] (s1_4) -- (f_1_2);

% Fork from s0_4 (3.2, 4)
\node[node] (f_0_1) at (4.0, 5) {}; % Moved from (3.4, 5) to (4.0, 5) to align with the line
%\node[node] (f_0_2) at (4.2, 5) {}; % Removed the rightmost node
\draw[thick] (s0_4) -- (f_0_1);
%\draw[thick] (s0_4) -- (f_0_2); % Removed the edge to the node

% 4. Add Dots (\cdots)
\tikzstyle{dots}=[font=\large]

% Level 1 dots
\node[dots] at (-0.1, 1) {$\cdots$};
\node[dots] at (1.8, 1) {$\cdots$};

% Level 2 dots
\node[dots] at (-1.1, 2) {$\cdots$};
\node[dots] at (0.7, 2) {$\cdots$};
\node[dots] at (2.6, 2) {$\cdots$};

% Level 3 dots
\node[dots] at (-2.1, 3) {$\cdots$};
\node[dots] at (-0.3, 3) {$\cdots$};
\node[dots] at (1.5, 3) {$\cdots$};
\node[dots] at (3.4, 3) {$\cdots$};

% Level 4 dots
\node[dots] at (-3.1, 4) {$\cdots$};
\node[dots] at (-1.3, 4) {$\cdots$};
\node[dots] at (0.5, 4) {$\cdots$};
\node[dots] at (2.3, 4) {$\cdots$};
\node[dots] at (4.2, 4) {$\cdots$};

% Level 5 dots (between groups)
%\node[dots] at (-4.4, 5) {$\cdots$};
%\node[dots] at (-2.5, 5) {$\cdots$};
%\node[dots] at (-0.7, 5) {$\cdots$};
%\node[dots] at (1.1, 5) {$\cdots$};
%\node[dots] at (2.9, 5) {$\cdots$};
\node[dots] at (5.0, 5) {$\cdots$}; % Adjusted after removing the rightmost node

% Level 5 dots (within forks)
%\node[dots] at (-3.4, 5) {$\cdots$};
\node[dots] at (-1.6, 5) {$\cdots$};
\node[dots] at (0.2, 5) {$\cdots$};
%\node[dots] at (2.0, 5) {$\cdots$};
%\node[dots] at (3.8, 5) {$\cdots$}; % Commented out after removing the rightmost node

\end{tikzpicture}
\end{center}

Fix a valuation $V:W\times \mathrm{Fml}\rightarrow 2$ on $(W, R)$. We show that $A^{(2k+1)}\leq_{T}V$. We take a $\Sigma^{0}_{2k+1}$-formula $\varphi(x, A)$ such that for all $n$, $\varphi(n, A)\Leftrightarrow n\in A^{(2k+1)}$. Then, there exists $e$ such that $\varphi(x, A)=\exists n_{1}\forall n_{2}\exists n_{3}\dots\forall n_{2k}\exists n_{2k+1}\theta_{e}(\vec{n}, x, A)$. Then, for $\langle{n, e\rangle}\in W$,
\begin{align*}
V(\langle{n, e\rangle}, \Diamond(\Box\Diamond)^{k}\top)=1&\Leftrightarrow\exists n_{1}\forall n_{2}\exists n_{3}\dots\forall n_{2k}\exists n_{2k+1}\\&(\langle{n, e, n_{1}, \dots, n_{2k}\rangle}R\langle{n, e, n_{1}, \dots, n_{2k},n_{2k+1}\rangle})\\
&\Leftrightarrow\exists n_{1}\forall n_{2}\exists n_{3}\dots\forall n_{2k}\exists n_{2k+1}\theta_{e}(n_{1}, \dots, n_{2k}, n_{2k+1}, n, A)\\
&\Leftrightarrow\varphi(n, A)\\
&\Leftrightarrow n\in A^{(2k+1)}.
\end{align*}
Thus, $V$ computes $A^{(2k+1)}$.
\end{proof}

\begin{cor}\label{2k+1resvalu}
Let $A$ be a set and $k\in\omega$. Then there exists a $A$-computable frame $F=(W, R)$ such that for any $(2k+1)$-restricted valuation $V:W\times \mathrm{Fml}[2k+1]\rightarrow 2$ on $(W, R)$, $V$ computes $A^{(2k+1)}$.
\end{cor}

\begin{proof}
By Theorem~\ref{comframepr}.
\end{proof}

\begin{cor}\label{vcompomega}
Let $A$ be a set. Then there exists a $A$-computable frame $F=(W, R)$ such that for any valuation $V:W\times \mathrm{Fml}\rightarrow 2$ on $(W, R)$, $V$ computes $A^{(\omega)}$.
\end{cor}

\begin{proof}
This can be proved by a construction of a frame similar to that in Theorem~\ref{comframepr}.
\end{proof}

By the above, the following holds.

\begin{cor}[$\mathrm{RCA}_0$]
VEL implies $\mathrm{ACA}^{+}_0$.
\end{cor}

\begin{proof}
This is obtained by formalizing the arguments of Theorem~\ref{comframepr} and Corollary~\ref{vcompomega} within $\mathrm{ACA}_0$, which is possible since VEL implies $\mathrm{ACA}_0$ over $\mathrm{RCA}_0$ (by Proposition~\ref{velraca}).
\end{proof}

\begin{cor}[$\mathrm{RCA}_0$]
Restricted VEL implies $\mathrm{ACA}^{\prime}_0$.
\end{cor}

\subsection{Frame definability for modal propositional logic}

Next, we discuss frame definability for modal propositional logic. In particular, we consider frame definability for Geach axioms and $\mathbf{GL}$. First, we introduce Geach axioms. The proof of frame definability for Geach axioms is obtained by formalizing the proof in~\cite{MR556867} within $\mathrm{ACA}^{+}_0$.

\begin{dfn}[$\mathrm{RCA}_0$]
Let $(W, R)$ be a frame, $n\in\mathbb{N}$, and $u, v\in W$, and let $\mathrm{Sec}^{=n}(W)$ be the set of all sequences $\sigma$ of elements of $W$ such that $lh(\sigma)=n$, and for all $i, j$ with $i<j<n$, $\sigma_{i}R\sigma_{j}$. Then we denote $u\rightsquigarrow^{n}v$ by $\exists\sigma\in\mathrm{Sec}^{=n}(W)(\sigma_{0}=u\land\sigma_{lh(\sigma)-1}=v)$.
\end{dfn}

\begin{prop}[$\mathrm{RCA}_0$]\label{boxneq}
Let $(W, R, V)$ be a Kripke model, $\varphi\in\mathrm{Fml}$, $w\in W$, and $n\in\mathbb{N}$. Then the following holds:
\begin{align*}
V(w, \Box^{n}\varphi)=1&\Leftrightarrow\forall u\in W((w\rightsquigarrow^{n}u)\rightarrow V(u, \varphi)=1),\\
V(w, \Diamond^{n}\varphi)=1&\Leftrightarrow\exists u\in W((w\rightsquigarrow^{n}u)\land V(u, \varphi)=1).
\end{align*} 
\end{prop}

\begin{proof}
We show the proof only for the case of $\Box^{n}\varphi$. Let $(W, R, V)$ be a Kripke model, $\varphi\in\mathrm{Fml}$. Then the following holds: for all $n\in\mathbb{N}$ and $w\in W$,
\begin{align*}
\forall u\in W((w\rightsquigarrow^{n}u)\rightarrow V(u, \varphi)=1)\Leftrightarrow\forall\sigma\in\mathrm{Sec}^{=n}(W)(\sigma_{0}=w\rightarrow V(\sigma_{lh(\sigma)-1}, \varphi)=1).
\end{align*} 

Then we show the following statement by $\Sigma^{0}_{1}$-induction on $n$:
\begin{align*}
\forall w\in W\bigl(V(w, \Box^{n}\varphi)=1\Rightarrow\forall u\in W((w\rightsquigarrow^{n}u)\rightarrow V(u, \varphi)=1)\bigr).
\end{align*} 

The case of $n=0$ is obvious. Assume that the case of $n$ holds. Fix $w\in W$ with $V(w, \Box^{n+1}\varphi)=1$. Take $u\in W$ with $w\rightsquigarrow^{n+1}u$. There exists $w_{0}\in W$ such that $wRw_{0}$ and $w_{0}\rightsquigarrow^{n}u$. Since $V(w, \Box^{n+1}\varphi)=1$ holds, $V(w_{0}, \Box^{n}\varphi)=1$ holds. By assumption, $\forall v\in W((w_{0}\rightsquigarrow^{n}v)\rightarrow V(v, \varphi)=1)$. Thus, $V(u, \varphi)=1$ holds, so the case of $n+1$ holds. This follows by induction.
\vspace{2mm}

Next, we show that the right side of the statement implies the left side. To derive a contradiction, we assume the following: there exists $m\in\mathbb{N}$ and $w_{0}\in W$ such that 
\begin{align*}
V(w_{0}, \Box^{m}\varphi)=0\land\forall u\in W((w_{0}\rightsquigarrow^{m}u)\rightarrow V(u, \varphi)=1).
\end{align*} 

Then we show the following statement by $\Sigma^{0}_{1}$-induction on $k\leq m$:
\begin{align*}
\forall i(m-k\leq i\leq m\rightarrow(\exists u_{\star}\in W)\bigl((w_{0}\rightsquigarrow^{m-i}u_{\star})\land V(u_{\star}, \Box^{i}\varphi)=0\bigr)).
\end{align*} 

If $k=0$, then $i=m$ and $V(w_{0}, \Box^{m}\varphi)=0$. Thus, the case of $k=0$ holds. Assume that the case of $k$ holds. By assumption, it is enough to consider the case of $i=m-(k+1)$. Then, by assumption, there exists $u_{\star}\in W$ such that $(w_{0}\rightsquigarrow^{k}u_{\star})$ and $V(u_{\star}, \Box^{m-k}\varphi)=0$ hold. Thus, there exists $v_{\star}\in W$ such that $u_{\star}Rv_{\star}$ and $V(v_{\star}, \Box^{m-(k+1)}\varphi)=0$, so $V(v_{\star}, \Box^{i}\varphi)=0$. Also, $w_{0}\rightsquigarrow^{k+1}v_{\star}$, that is, $w_{0}\rightsquigarrow^{m-i}v_{\star}$. Thus, the case of $k+1$ holds. By induction, we show it. However, for $k=m$ and $i=0$, $(\exists u_{\star}\in W)\bigl((w_{0}\rightsquigarrow^{m}u_{\star})\land V(u_{\star}, \varphi)=0\bigr)$ holds. This is a contradiction.
\end{proof}

\begin{dfn}[Geach axiom, $\mathrm{RCA}_0$~\cite{MR556867}]
Let $i, j, m, n\in\mathbb{N}$ and $\varphi\in\mathrm{Fml}$. Then we define the modal formula $\mathrm{Ga}$ (Geach axiom):
\begin{align*}
\mathrm{Ga}(i, j, m, n;\varphi)\equiv\Diamond^{i}\Box^{m}\varphi\rightarrow\Box^{j}\Diamond^{n}\varphi.
\end{align*}
\end{dfn}

In what follows, we assume that the parameters $i, j, m, n$ in the Geach axiom $\mathrm{Ga}(i, j, m, n; \varphi)$ satisfy neither $i=n=0$ nor $m=j=0$ holds. The following proposition is part of the proof of the soundness theorem for Geach axioms. As shown below, frame definability for Geach axioms is equivalent to $\mathrm{ACA}^{+}_0$ over $\mathrm{RCA}_0$, but this part is provable in $\mathrm{RCA}_0$.

\begin{prop}[$\mathrm{RCA}_0$~\cite{MR556867}]\label{fdeasy}
Let $F=(W, R)$ be a frame and $i, j, m, n\in\mathbb{N}$. Assume that the following holds: $\forall x\forall y\forall z\in W\bigl((x\rightsquigarrow^{i}y)\land(x\rightsquigarrow^{j}z)\rightarrow\exists u\in W((y\rightsquigarrow^{m}u)\land(z\rightsquigarrow^{n}u))\bigr)$. Then for all $\varphi\in\mathrm{Fml}$, $F\Vdash \mathrm{Ga}(i, j, m, n;\varphi)$.
\end{prop}

\begin{proof}
Let $\varphi\in\mathrm{Fml}$ and $(W, R, V)$ be a Kripke model. Fix $x\in W$, and assume that $V(x, \Diamond^{i}\Box^{m}\varphi)=1$. Then, by Proposition~\ref{boxneq}, there exists $y\in W$ such that $x\rightsquigarrow^{i}y$ and $V(y, \Box^{m}\varphi)=1$. Take $z\in W$ with $x\rightsquigarrow^{j}z$. By assumption, we have $u\in W$ such that $y\rightsquigarrow^{m}u$ and $z\rightsquigarrow^{n}u$. Then $V(z, \Diamond^{n}\varphi)=1$ by $V(u, \varphi)=1$ and Proposition~\ref{boxneq}. Thus, $V(x, \Box^{j}\Diamond^{n}\varphi)=1$ holds, so $F\Vdash \mathrm{Ga}(i, j, m, n;\varphi)$.
\end{proof}

\begin{thm}[Frame definability for Geach axioms, $\mathrm{ACA}^{+}_0$~\cite{MR556867}]\label{fdga}
Let $F=(W, R)$ be a frame and $i, j, m, n\in\mathbb{N}$. Then the following statements are equivalent:
\begin{enumerate}
    \item $\forall x\forall y\forall z\in W\bigl((x\rightsquigarrow^{i}y)\land(x\rightsquigarrow^{j}z)\rightarrow\exists u\in W((y\rightsquigarrow^{m}u)\land(z\rightsquigarrow^{n}u))\bigr)$,
    \item $F\Vdash \mathrm{Ga}(i, j, m, n;\varphi)$ for all $\varphi\in\mathrm{Fml}$.
\end{enumerate}
\end{thm}

\begin{proof}
$(1)\Rightarrow(2)$: By Proposition~\ref{fdeasy}.
\vspace{1mm}

$(2)\Rightarrow(1)$: Fix $p_{0}\in\mathrm{Prop}$. Assume that there exists $a, b, c\in W$ such that $(a\rightsquigarrow^{i}b)\land(a\rightsquigarrow^{j}c)$ and $\forall u\in W((b\not\rightsquigarrow^{m}u)\lor(c\not\rightsquigarrow^{n}u))$. Then we define $v:W\times\mathrm{Prop}\rightarrow2$ by the following: for all $w\in W$,
\begin{align*}
&v(w, p_{0})=1\Leftrightarrow b\rightsquigarrow^{m}w,\\
&v(w, p)=1,\, \text{for all $p\in\mathrm{Prop}\setminus\{p_{0}\}$}.
\end{align*}

By VEL (that is, $\mathrm{ACA}^{+}_0$), we have the valuation $V: W\times\mathrm{Fml}\rightarrow2$ on $(W, R)$ such that $V$ is an extension of $v$. Then for all $d\in W$ with $b\rightsquigarrow^{m}d$, $V(d, p_{0})=1$. Thus, $V(b, \Box^{m}p_{0})=1$. Since $a\rightsquigarrow^{i}b$, $V(a, \Diamond^{i}\Box^{m}p_{0})=1$. By assumption, $a\rightsquigarrow^{j}c$ and $\forall u\in W((c\rightsquigarrow^{n}u)\rightarrow(b\not\rightsquigarrow^{m}u))$. Thus, $V(c, \Box^{n}\neg p_{0})=1$, so $V(a, \Diamond^{j}\Box^{n}\neg p_{0})=1$, that is, $V(a, \Box^{j}\Diamond^{n}p_{0})=0$. Thus, $V(a,\mathrm{Ga}(i, j, m, n;p_{0}))=0$, so $F\not\Vdash \mathrm{Ga}(i, j, m, n;p_{0})$.
\end{proof}

Regarding frame definability, one can consider the following weaker version. 

\begin{cor}[Weak frame definability for Geach axioms, $\mathrm{ACA}^{\prime}_0$]\label{wfdgeach}
Let $F=(W, R)$ be a frame and $i, j, m, n\in\mathbb{N}$ with $k=i+j+m+n$. Then the following statements are equivalent:
\begin{enumerate}
    \item $\forall x\forall y\forall z\in W\bigl((x\rightsquigarrow^{i}y)\land(x\rightsquigarrow^{j}z)\rightarrow\exists u\in W((y\rightsquigarrow^{m}u)\land(z\rightsquigarrow^{n}u))\bigr)$,
    \item $(F, V)\Vdash\mathrm{Ga}(i, j, m, n;p)$ for every $k$-restricted valuation $V$ on $F$ and $p\in\mathrm{Prop}$.
\end{enumerate}
\end{cor}

\begin{proof}
The part of the argument for Theorem~\ref{fdga} that does not rely on VEL can be carried out within $\mathrm{RCA}_0$. On the other hand, for the part that requires VEL, Corollary~\ref{pextend2} implies that it suffices to construct a $k$-restricted valuation over $\mathrm{ACA}^{\prime}_0$.
\end{proof}

%$\mathbf{GL}$ is the smallest normal modal logic which contains $\mathbf{K}\cup\{\Box p\rightarrow\Box\Box p\}$.

%\begin{prop}[$\mathrm{RCA}_0$~\cite{10.1007/978-3-031-95908-0_32}]\label{gl4}
%$\mathrm{Pbl}_{\mathbf{GL}}(\emptyset, \Box p\rightarrow\Box\Box p)$ holds.
%\end{prop}

As an example of a frame-definability result not expressible by a Geach axiom, we consider the frame definability of the Gödel–Löb logic, $\mathbf{GL}$. The proof of frame definability for $\mathbf{GL}$ is obtained by formalizing the proof in~\cite{MR1837791} within $\mathrm{ACA}^{+}_0$.

\begin{prop}[Frame definability for $\mathbf{GL}$, $\mathrm{ACA}^{+}_0$~\cite{MR1837791}]\label{fdgl}
Let $F=(W, R)$ be a frame. Then the following statements are equivalent:
\begin{enumerate}
    \item $F$ is transitive and $W$ has no infinite ascending sequences by $R$,
    \item $F\Vdash\Box(\Box\varphi\rightarrow\varphi)\rightarrow\Box\varphi$ for all $\varphi\in\mathrm{Fml}$.
\end{enumerate}
\end{prop}

\begin{proof}
$(1)\Rightarrow(2)$: Let $F=(W, R)$ be a transitive frame such that $W$ has no infinite ascending sequences by $R$. Take a model $M=(W, R, V)$ and a modal formula $\varphi\in\mathrm{Fml}$. Fix $w_{0}\in W$ with $V(w_{0}, \Box(\Box\varphi\rightarrow\varphi))=1$. To derive a contradiction, we assume that $V(w_{0}, \neg\Box\varphi)=1$. Then, we define an infinite ascending sequence $\{w_{i}\}_{i}$ by $R$ such that for all $i$, $w_{0}Rw_{i+1}$ and $V(w_{i},\neg\Box\varphi)=1$ as follows: 

Assume that we have already defined $w_{i}$. Then there exists $w_{i+1}\in W$ such that $w_{i}Rw_{i+1}$ and $V(w_{i+1}, \neg\varphi)=1$. By transitivity of $R$, $w_{0}Rw_{i+1}$. By assumption, $V(w_{0}, \Box(\Box\varphi\rightarrow\varphi))=1$. Thus, $V(w_{i+1},\Box\varphi\rightarrow\varphi)=1$, so, $V(w_{i+1}, \neg\Box\varphi)=1$ holds. Then, we define the infinite ascending sequence $\{w_{i}\}_{i}$ by $R$. However, it is a contradiction since $W$ has no infinite ascending sequences by $R$. Thus, $V(w_{0}, \Box(\Box\varphi\rightarrow\varphi)\rightarrow\Box\varphi)=1$. 
\vspace{2mm}

$(2)\Rightarrow(1)$: Assume that $F$ is not transitive. By Theorem~\ref{fdga}, there exists a valuation $V$ on $(W, R)$, $w\in W$ and $\varphi\in\mathrm{Fml}$ such that $V(w, \mathrm{Ga}(0, 2, 1, 0;\varphi))=0$, that is, $V(w, \Box\varphi\rightarrow\Box\Box\varphi)=0$. By the soundness theorem for modal logic (see~\cite{10.1007/978-3-031-95908-0_32}), $V(w, \Box(\Box\varphi\rightarrow\varphi)\rightarrow\Box\varphi)=0$. 

Assume that $W$ has an infinite ascending sequence $\{w_{i}\}_{i}$ by $R$. Fix $p_{0}\in\mathrm{Prop}$. Then we define $v:W\times\mathrm{Prop}\rightarrow2$ by the following: for all $w\in W$,
\begin{align*}
&v(w, p_{0})=1\Leftrightarrow\forall i(w\neq w_{i}),\\
&v(w, p)=1,\, \text{for all $p\in\mathrm{Prop}\setminus\{p_{0}\}$}.
\end{align*}

Such a $v$ exists by arithmetical comprehension. By VEL (that is, $\mathrm{ACA}^{+}_0$), we have a valuation $V: W\times\mathrm{Fml}\rightarrow2$ on $(W, R)$ such that $V$ extends $v$. Then take $u\in W$ with $w_{0}Ru$. Assume that $V(u, \neg p_0)=1$. Then there exists $i$ such that $u=w_{i}$. By $uRw_{i+1}$ and $V(w_{i+1}, \neg p_0)=1$, $V(u, \neg\Box p_0)=1$. Thus, $V(u, \neg p_0\rightarrow\neg\Box p_0)=1$, so, $V(u, \Box p_0\rightarrow p_0)=1$. Then, we have $V(w_{0}, \Box(\Box p_0\rightarrow p_0))=1$. By $V(w_{1}, \neg p_0)=1$, $V(w_{0}, \neg\Box p_0)=1$. Thus, $V(w_{0}, \Box(\Box p_0\rightarrow p_0)\rightarrow\Box p_0)=0$. 
\end{proof}

\begin{cor}[Weak frame definability for $\mathbf{GL}$, $\mathrm{ACA}_0$]\label{wfdgl}
Let $F=(W, R)$ be a frame. Then the following statements are equivalent:
\begin{enumerate}
    \item $F$ is transitive and $W$ has no infinite ascending sequences by $R$,
    \item $(F, V)\Vdash\Box(\Box p\rightarrow p)\rightarrow\Box p$ for every $3$-restricted valuation $V$ on $F$ and $p\in\mathrm{Prop}$.
\end{enumerate}
\end{cor}

\begin{proof}
The part of the argument for Proposition~\ref{fdgl} that does not rely on VEL can be carried out within $\mathrm{ACA}_0$. On the other hand, for the part that requires VEL, Corollary~\ref{pxtendstandard} implies that it suffices to construct a $3$-restricted valuation over $\mathrm{ACA}_0$.
\end{proof}

Consequently, the relationship between VEL and frame definability can be summarized as follows.

\begin{thm}\label{aca+equiv}
The following statements are equivalent over $\mathrm{RCA}_0$:
\begin{enumerate}
    \item $\mathrm{ACA}^{+}_0$,
    \item The Valuation Extension Lemma,
    \item Frame definability for Geach axioms,
    \item Frame definability for $\mathbf{GL}$,
    \item For each frame $F=(W, R)$, there exists a valuation $V$ on $(W, R)$.
\end{enumerate}
\end{thm}

\begin{proof}
$(1)\Rightarrow(2)$: By Theorem~\ref{pextend}. 
\vspace{1mm}

$(2)\Rightarrow(3)$: By Theorem~\ref{fdga}. 
\vspace{1mm}

$(2)\Rightarrow(4)$: By Proposition~\ref{velraca} and Proposition~\ref{fdgl}.
\vspace{1mm}

$(3)\Rightarrow(5)$: Let $F=(W, R)$ be a frame. We may assume that $i, j>0$. Let $W^{\prime}=\{x, y_{1}, \dots, y_{i}, z_{1}, \dots, z_{j}\}$ be new worlds, and let $R^{\prime}\subseteq W^{\prime}\times W^{\prime}$ be a new relation as follows:
\begin{align*}
R^{\prime}=\{\langle{x, y_{1}\rangle}, \langle{y_{1}, y_{2}\rangle},\dots, \langle{y_{i-1}, y_{i}\rangle}, \langle{x, z_{1}\rangle}, \langle{z_{1}, z_{2}\rangle}, \dots, \langle{z_{j-1}, z_{j}\rangle}\}.
\end{align*}

Now, we define the new frame $F^{\star}=(W\sqcup W^{\prime}, R\sqcup R^{\prime})$. Then there exists $x, y_{i}, z_{j}$ such that $x\rightsquigarrow^{i}y_{i}$ and $x\rightsquigarrow^{j}z_{j}$ and for all $u\in W$, $y_{i}\not\rightsquigarrow^{m}u$ or $z_{j}\not\rightsquigarrow^{n}u$. Thus, by frame definability for Geach axioms, there exists a valuation $V^{\star}$ on $F^{\star}$. Then, $V=V^{\star}\upharpoonright(W\times\mathrm{Fml})$ is a valuation on $(W, R)$.
\vspace{1mm}

$(4)\Rightarrow(5)$: This can be shown by the same method used for the implication $(3)\Rightarrow(5)$.
\vspace{1mm}

$(5)\Rightarrow(1)$: This can be shown by formalizing the proof of Corollary~\ref{vcompomega} in $\mathrm{RCA}_0$ (or $\mathrm{ACA}_0$).
\end{proof}

\begin{thm}\label{acaprimerveldouchi}
The following statements are equivalent over $\mathrm{RCA}_0$:
\begin{enumerate}
    \item $\mathrm{ACA}^{\prime}_0$,
    \item The Restricted Valuation Extension Lemma,
    \item Weak frame definability for Geach axioms,
    \item For each frame $F=(W, R)$ and $k\in\mathbb{N}$, there exists a $k$-restricted valuation $V$ on $(W, R)$.
\end{enumerate}
\end{thm}

\begin{proof}
$(1)\Rightarrow(2)$: By Corollary~\ref{pextend2}. 
\vspace{1mm}

$(2)\Rightarrow(3)$: By Corollary~\ref{wfdgeach}.
\vspace{1mm}

$(3)\Rightarrow(4)$: This can be shown by the implication $(3)\Rightarrow(5)$ of Theorem~\ref{aca+equiv}. 
\vspace{1mm}

$(4)\Rightarrow(1)$: By Corollary~\ref{2k+1resvalu}. 
\end{proof}

\begin{thm}
The following statements are equivalent over $\mathrm{RCA}_0$:
\begin{enumerate}
    \item $\mathrm{ACA}_0$,
    \item Weak frame definability for $\mathbf{GL}$,
    \item For each frame $F=(W, R)$, there exists a $1$-restricted valuation $V$ on $(W, R)$.
\end{enumerate}
\end{thm}

\begin{proof}
$(1)\Rightarrow(2)$: By Corollary~\ref{wfdgl}. 
\vspace{1mm}

$(2)\Rightarrow(3)$: This can be shown by the implication $(4)\Rightarrow(5)$ of Theorem~\ref{aca+equiv}. 
\vspace{1mm}

$(3)\Rightarrow(1)$: By Proposition~\ref{wwwvelaca}. 
\end{proof}

Moreover, if a frame is image-finite, we obtain the following result.

\begin{thm}
The following statements are equivalent over $\mathrm{RCA}_0$:
\begin{enumerate}
    \item $\mathrm{ACA}_0$,
    \item The Valuation Extension Lemma for image-finite frames,
    \item Frame definability for Geach axioms on image-finite frames,
    \item Frame definability for $\mathbf{GL}$ on image-finite frames,
    \item For each image-finite frame $F=(W, R)$, there exists a valuation $V$ on $(W, R)$.
\end{enumerate}
\end{thm}

\begin{proof}
$(1)\Rightarrow(2)$: Let $F=(W, R)$ be an image-finite frame and an assignment $v:W\times\mathrm{Prop}\rightarrow2$. We enumerate elements of $W$ as follows: $W=\{w_{i}\}_{i}$. We define $A\subseteq\mathbb{N}$ as follows: for all $m$ and $w\in W$,
\begin{align*}
\langle{w, m\rangle}\in A:\Leftrightarrow\forall i(i>m\rightarrow \neg(wRw_{i})).
\end{align*}  
Such an $A$ exists by arithmetical comprehension. For each $w\in W$, we let $A_{w}=\{m\mid\langle{w,m\rangle}\in A\}$. Since $F$ is image-finite, $A_{w}$ is not empty for each $w\in W$. 
\vspace{1mm}

Then, we construct a valuation $V: W \times \mathrm{Fml} \to 2$ extending $v$ in a bottom-up manner, using $F\oplus A\oplus v$ as an oracle. The construction is similar to that in Theorem~\ref{pextend}. Step $0$ is carried out in the same way as in Theorem~\ref{pextend}; we describe Step $k$.
\vspace{2mm}

Step $k$: We assume that $V_{k-1}: W \times \mathrm{Fml}[k-1] \to 2$ is defined and $e_{k-1}$ is an index of $V_{k-1}$. Fix $w\in W$ and $\varphi\in\mathrm{Fml}[k-1]$. Then, using $A$ as an oracle, we search for the least element of $A_{w}$, say $\langle{w, m\rangle}$, and find all elements of $W$ bounded by $m$ that belong to $wR$. Let $wR=\{w_{i_{0}}, \dots, w_{i_{s}}\}$ $(i_{s}\leq m)$. We then determine whether $\forall j\leq s(V_{k-1}(w_{i_{j}}, \varphi)=1)$ holds, and set $V_{k}(w, \Box\varphi)=1$ if this is the case, and $0$ otherwise. Using $F \oplus A\oplus v$ as an oracle, we compute an index $e_{k}$ for the $k$-restricted valuation $V_{k} : W \times \mathrm{Fml}[k] \to 2$ defined in this way.
\vspace{1mm}

This construction can be carried out in $\mathrm{ACA}_0$ using $F \oplus A\oplus v$ as an oracle. In this construction, we have the indices $E=\{e_{k}\}_{k}$. The rest is the same as in Theorem~\ref{pextend}.
\vspace{1mm}

$(2)\Rightarrow(3)$: This can be shown by the implication $(2)\Rightarrow(3)$ of Theorem~\ref{aca+equiv}. 
\vspace{1mm}

$(2)\Rightarrow(4)$: This can be shown by the implication $(2)\Rightarrow(4)$ of Theorem~\ref{aca+equiv}. 
\vspace{1mm}

$(3)\Rightarrow(5)$: This can be shown by the implication $(3)\Rightarrow(5)$ of Theorem~\ref{aca+equiv}. 
\vspace{1mm}

$(4)\Rightarrow(5)$: This can be shown by the implication $(4)\Rightarrow(5)$ of Theorem~\ref{aca+equiv}. 
\vspace{1mm}

$(5)\Rightarrow(1)$: By Proposition~\ref{acaeqijf}, it suffices to show that for an injective function $f$, there exists a range set of $f$. We take an injective function $f:\mathbb{N}\rightarrow\mathbb{N}$. Let $W=\mathbb{N}$ and $R=\{\langle{n, m\rangle}\mid f(m)=n\}$. Since $f$ is injective, $F=(W, R)$ is image-finite. The remainder of the proof follows by an argument similar to that of Proposition~\ref{wwwvelaca}.
\end{proof}

\begin{cor}
The following statements are provable in $\mathrm{RCA}_0$:
\begin{enumerate}
    \item The Valuation Extension Lemma for finite frames,
    \item Frame definability for Geach axioms on finite frames,
    \item Frame definability for $\mathbf{GL}$ on finite frames,
    \item For each finite frame $F=(W, R)$, there exists a valuation $V$ on $(W, R)$.
\end{enumerate}
\end{cor}

\subsection{Frame definability for modal predicate logic}
Next, we define variable-domain Kripke models for (first-order) modal predicate logic. The language $\mathcal{L}_{\mathrm{MP}}$ consists of the modal operator $\Box$ and the standard first-order symbols, including countably many relation symbols, but no function symbols. Suppose that $\mathcal{L}_{\mathrm{MP}}$ contains countably many unary predicates $P_{0}(x), P_{1}(x), \dots$. We identify terms and formulas with their G\"odel numbers under a fixed primitive recursive coding. Then let $\mathrm{Atm}_{\mathrm{MP}}$, $\mathrm{Fml}_{\mathrm{MP}}$, and $\mathrm{Snt}_{\mathrm{MP}}$ be the sets of $\mathcal{L}_{\mathrm{MP}}$-atomic formulas, $\mathcal{L}_{\mathrm{MP}}$-formulas, and $\mathcal{L}_{\mathrm{MP}}$-sentences, respectively. Our formalization of the semantics over $\mathrm{RCA}_0$ is based on~\cite{MR3558757, MR2517689, MR2010178}.

\begin{dfn}[$\mathrm{RCA}_0$~\cite{MR3558757, MR2517689, MR2010178}]
A \emph{variable-domain Kripke model for modal predicate logic} is a tuple $\mathcal{M}=(W, R, D, V)$ satisfying the following conditions:
\begin{itemize}
    \item $W\subseteq\mathbb{N}$ is a non-empty set,
    \item $R$ is a binary relation on $W$, i.e., $R\subseteq W\times W$,
    \item $D\subseteq W\times\mathbb{N}$ is a non-empty set,
    \item For any $w\in W$, $D_{w}=\{d\in\mathbb{N}\mid\langle{w, d\rangle}\in D\}$ is a non-empty set,
    \item $T_M$ and $S_M$ are respectively the sets of closed terms and sentences of the expanded language $\mathcal{L}^{M}_{\mathrm{MP}}=\mathcal{L}_{\mathrm{MP}}\cup\{\overline{d}\mid d\in\mathbb{N}\}$ with new constant symbols $\overline{d}$ for each element $\{\overline{d}\mid d\in\mathbb{N}\}$.
    \item $A_M$ is the set of atomic formulas of the expanded language $\mathcal{L}^{M}_{\mathrm{MP}}$,
    \item %$V$ is a function which assigns a truth value to each pair of a propositional formula and an element of $W$, i.e., 
$V : W\times(T_{M}\cup S_{M})\rightarrow\mathbb{N}\cup\{0, 1\}$ satisfies the following conditions:
\begin{itemize}
    \item $V(w, t)\in D_{w}$ for any $w\in W$ and $t\in T_M$,
    \item $V(w,\varphi)\in\{0, 1\}$ for any $w\in W$ and $\varphi\in S_M$,
    \item $\bigl(\forall i\leq n(V(w, t_i)=V(w, t^{\prime}_i))\rightarrow V(w, R(t_1,\dots,t_n))=V(w, R(t^{\prime}_1,\dots,t^{\prime}_n))\bigr)$, \\where $R$ is a relation symbol, for any $w\in W$ and $t_1,  t^{\prime}_1, \dots  t_n, t^{\prime}_n\in T_M$,
    \item $V(w, \bot)=0$ for any $w\in W$,
    \item $V(w, \varphi\rightarrow\psi)=1-V(w, \varphi)(1-V(w, \psi))$ for any $w\in W$ and $\varphi, \psi\in S_{M}$,
   \item $V(w, \Box\varphi)=1\iff\forall v\in W(wRv\rightarrow V(v, \varphi)=1)$ for any $w\in W$ and $\varphi\in S_{M}$,
    \item $V(w, \forall x\varphi(x))=1\iff\forall d\in D_{w}(V(w, \varphi(\overline{d}))=1)$ for any $w\in W$ and $\varphi(\overline{d})\in S_{M}$. 
\end{itemize}
\end{itemize}
A tuple $(W, R, D)$ is called a \emph{frame for modal predicate logic}, and a function $V$ is called a \emph{valuation for modal predicate logic} on $(W, R, D)$. We often identify $d\in D_{w}$ with $\overline{d}$ for each $w\in W$.
\end{dfn}

We also use the following notation.
Let $\mathcal{F}$ be a frame for modal predicate logic and let $\varphi\in\mathrm{Snt}_{\mathrm{MP}}$. We define
\begin{align*}
%&M, w\Vdash\varphi\equiv V(w, \varphi)=1, \quad \quad M\Vdash\varphi\equiv(\forall w\in W)(V(w, \varphi)=1),\\
\mathcal{F}\Vdash\varphi\equiv(\forall V\text{: valuation for modal predicate logic on $\mathcal{F}$})(\forall w\in W)(V(w, \varphi)=1).
\end{align*}

As in the case of modal propositional logic, VEL for modal predicate logic can also be proved in $\mathrm{ACA}^{+}_0$.

\begin{prop}[The Valuation Extension Lemma for modal predicate logic, $\mathrm{ACA}^{+}_0$]\label{velmplaca*}
For each frame for modal predicate logic $\mathcal{F}=(W, R, D)$ and $v: W\times A_{M}\rightarrow2$, there exists a valuation $V$ for modal predicate logic on $\mathcal{F}$ such that $V$ is an extension of $v$.
\end{prop}

\begin{proof}
Similarly to the construction of the evaluation function in Theorem~\ref{pextend}, the valuation for modal predicate logic which extends $v$ can also be computed from $(\mathcal{F}\oplus v)^{(\omega)}$.
\end{proof}

We now introduce the Barcan formula and the Converse Barcan formula. In modal predicate logic, frame definability holds for these formulas.

\begin{dfn}[Barcan formula and Converse Barcan formula, $\mathrm{RCA}_0$~\cite{MR3558757}]
Let $\varphi(x)\in\mathrm{Fml}_{\mathrm{MP}}$ which has no free number variables except $x$. We define the formulas $\mathrm{BF}\varphi$ (Barcan formula) and $\mathrm{CBF}\varphi$ (Converse Barcan formula):
\begin{align*}
&\mathrm{BF}\varphi\equiv\forall x\Box\varphi(x)\rightarrow\Box\forall x\varphi(x),\\
&\mathrm{CBF}\varphi\equiv\Box\forall x\varphi(x)\rightarrow\forall x\Box\varphi(x).
\end{align*}
\end{dfn}

\begin{prop}[Frame definability for $\mathrm{BF}$, $\mathrm{ACA}^{+}_0$~\cite{MR3558757}]\label{bffdmpl}
Let $\mathcal{F}=(W, R, D)$ be a frame for modal predicate logic. Then the following statements are equivalent:
\begin{enumerate}
    \item $\mathcal{F}$ is a domain-decreasing frame, that is, $\forall w\forall v\in W(wRv\rightarrow D_{v}\subseteq D_{w})$,
    \item $\mathcal{F}\Vdash\mathrm{BF}\varphi$ for all $\varphi(x)\in\mathrm{Fml}_{\mathrm{MP}}$ which has no free number variables except $x$.
\end{enumerate}
\end{prop}

\begin{proof}
$(1)\Rightarrow(2)$: Let $\mathcal{F}=(W, R, D)$ be a domain-decreasing frame and $\varphi(x)\in\mathrm{Fml}_{\mathrm{MP}}$ which has no free number variables except $x$. Fix $w\in W$, and assume that $V(w, \forall x\Box\varphi(x))=1$. Take $w^{\prime}\in wR$ and $d\in D_{w^{\prime}}$. Since $\mathcal{F}$ is domain-decreasing, $d\in D_{w}$. Thus, $V(w, \Box\varphi(d))=1$ and $V(w^{\prime}, \varphi(d))=1$. Then $V(w, \Box\forall x\varphi(x))=1$. Therefore $V(w, \mathrm{BF}\varphi)=1$. 
\vspace{1mm}

$(2)\Rightarrow(1)$: Fix $P_{0}(x)\in\mathrm{Atm}_{\mathrm{MP}}$. Assume that $\mathcal{F}$ is not domain-decreasing, that is, there exist $s, t\in W$ such that $sRt$ and $D_{t}\setminus D_{s}\not=\emptyset$. Take $d\in D_{t}\setminus D_{s}$. Then we define $v:W\times A_{M}\rightarrow2$ by the following: for all $w, u\in W$ and $\langle{u, a\rangle}\in D$,
\begin{align*}
&v(s, P_{0}(a))=1\Leftrightarrow a\in D_{s},\\
&v(t, P_{0}(a))=1\Leftrightarrow a\not=d,\\
&v(w, P_{0}(a))=1,\, \text{for all $w\not\in\{s, t\}$}.
\end{align*}

Atomic formulas other than $P_{0}(x)$ are evaluated to be true in every possible world. By VEL for modal predicate logic (that is, $\mathrm{ACA}^{+}_0$), we have the valuation $V: W\times(T_{M}\cup S_{M})\rightarrow2$ for modal predicate logic on $(W, R, D)$ such that $V$ is an extension of $v$. Then $V(s, \forall x\Box P_{0}(x))=1$, but $V(s, \Box\forall xP_{0}(x))=0$. Thus, $V(s, \mathrm{BF}P_0)=0$.
\end{proof}

\begin{cor}[Frame definability for $\mathrm{CBF}$, $\mathrm{ACA}^{+}_0$~\cite{MR3558757}]\label{cbffdmpl}
Let $\mathcal{F}=(W, R, D)$ be a frame for modal predicate logic. Then the following statements are equivalent:
\begin{enumerate}
    \item $\mathcal{F}$ is a domain-increasing frame, that is, $\forall w\forall v\in W(wRv\rightarrow D_{w}\subseteq D_{v})$,
    \item $\mathcal{F}\Vdash\mathrm{CBF}\varphi$ for all $\varphi(x)\in\mathrm{Fml}_{\mathrm{MP}}$ which has no free number variables except $x$.
\end{enumerate}
\end{cor}

\begin{proof}
$(1)\Rightarrow(2)$: Let $\mathcal{F}=(W, R, D)$ be a domain-increasing frame and $\varphi(x)\in\mathrm{Fml}_{\mathrm{MP}}$ which has no free number variables except $x$. Fix $w\in W$, and assume that $V(w, \Box\forall x\varphi(x))=1$. Take $d\in D_{w}$ and $w^{\prime}\in wR$. Since $\mathcal{F}$ is domain-increasing, $d\in D_{w^{\prime}}$. Since $V(w^{\prime}, \forall x\varphi(x))=1$, $V(w^{\prime}, \varphi(d))=1$. Thus, $V(w, \Box\varphi(d))=1$. Then $V(w, \forall x\Box\varphi(x))=1$. Therefore $V(w, \mathrm{CBF}\varphi)=1$. 
\vspace{1mm}

$(2)\Rightarrow(1)$: Fix $P_{0}(x)\in\mathrm{Atm}_{\mathrm{MP}}$. Assume that $\mathcal{F}$ is not domain-increasing, that is, there exist $s, t\in W$ such that $sRt$ and $D_{s}\setminus D_{t}\not=\emptyset$. Take $d\in D_{s}\setminus D_{t}$. Then we define $v:W\times A_{M}\rightarrow2$ by the following: for all $w, u\in W$ and $\langle{u, a\rangle}\in D$,
\begin{align*}
&v(w, P_{0}(a))=1\Leftrightarrow a\in D_{w},\, \text{for all $w\in\{s, t\}$}\\
&v(w, P_{0}(a))=1,\, \text{otherwise}.
\end{align*}

Atomic formulas other than $P_{0}(x)$ are evaluated to be true in every possible world. By VEL for modal predicate logic, we have the valuation $V: W\times(T_{M}\cup S_{M})\rightarrow2$ for modal predicate logic on $(W, R, D)$ such that $V$ is an extension of $v$. Then $V(s, \Box\forall x P_{0}(x))=1$, but $V(s, \forall x\Box P_{0}(x))=0$. Thus, $V(s, \mathrm{CBF}P_0)=0$.
\end{proof}

By the above, the following holds.

\begin{thm}
The following statements are equivalent over $\mathrm{RCA}_0$:
\begin{enumerate}
    \item $\mathrm{ACA}^{+}_0$,
    \item The Valuation Extension Lemma for modal predicate logic,
    \item Frame definability for $\mathrm{BF}$,
    \item Frame definability for $\mathrm{CBF}$,
    \item For each frame for modal predicate logic $\mathcal{F}=(W, R, D)$, there exists a valuation $V$ for modal predicate logic on $\mathcal{F}$.
\end{enumerate}
\end{thm}

\begin{proof}
$(1)\Rightarrow(2)$: By Proposition~\ref{velmplaca*}. 
\vspace{1mm}

$(2)\Rightarrow(3)$: By Proposition~\ref{bffdmpl}. 
\vspace{1mm}

$(2)\Rightarrow(4)$: By Corollary~\ref{cbffdmpl}.
\vspace{1mm}

$(3)\Rightarrow(5)$: Let $\mathcal{F}=(W, R, D)$ be a frame for modal predicate logic and let $a, b, c, d$ be new symbols that do not occur in $W$. We define $W^{\prime}=\{a, b\}$ and $R^{\prime}=\{\langle{a, b\rangle}\}$. Then we let $D^{\prime}=\{\langle{a, c\rangle}, \langle{b, d\rangle}\}$. 

Now, we define the new frame $\mathcal{F}^{\star}=(W^{\star}, R^{\star}, D^{\star})=(W\sqcup W^{\prime}, R\sqcup R^{\prime}, D\sqcup D^{\prime})$. Then there exists $a, b\in W^{\star}$ such that $aR^{\star}b$ and $D^{\star}_{b}\not\subseteq D^{\star}_{a}$. Thus, by frame definability for $\mathrm{BF}$, there exists a valuation $V^{\star}$ on $\mathcal{F}^{\star}$. Then, $V=V^{\star}\upharpoonright(W\times(T_{M}\cup S_{M}))$ is a valuation for modal predicate logic on $\mathcal{F}$.
\vspace{1mm}

$(4)\Rightarrow(5)$: This can be shown by the same method used for the implication $(3)\Rightarrow(5)$.
\vspace{1mm}

$(5)\Rightarrow(1)$: By Theorem~\ref{aca+equiv}, it is enough to show that for each frame $F=(W, R)$, there exists a valuation $V: W\times\mathrm{Fml}\rightarrow 2$ on $F$. Let $F=(W, R)$ be a frame and $c$ be a new symbol that does not occur in $W$. We define $D=\{\langle{w, c\rangle}\mid w\in W\}$. Then, for such $\mathcal{F}=(W, R, D)$, by the assumption, there exists a valuation $V^{\star}$ for modal predicate logic on $\mathcal{F}$. Next, for each $\varphi\in\mathrm{Fml}$, we define $\varphi^{\mathrm{MPc}}\in S_{M}$ as follows:
\begin{itemize}
    \item $(p_{i})^{\mathrm{MPc}}\equiv P_{i}(c)$ for all $i$,
    \item $(\bot)^{\mathrm{MPc}}\equiv\bot$,
    \item $(\varphi\rightarrow\psi)^{\mathrm{MPc}}\equiv(\varphi)^{\mathrm{MPc}}\rightarrow(\psi)^{\mathrm{MPc}}$,
    \item $(\Box\varphi)^{\mathrm{MPc}}\equiv\Box(\varphi)^{\mathrm{MPc}}$.
\end{itemize}

Then we define $V: W\times\mathrm{Fml}\rightarrow 2$; $V(w, \varphi)=V^{\star}(w, (\varphi)^{\mathrm{MPc}})$. This $V$ is a valuation on $F$.
\end{proof}

\section{The Valuation Extension Lemma for $\mathbf{CTL}$ and for $\mathbf{LTL}$}
Next, we further analyze VEL. In this section, we consider VEL for computation tree logic $\mathbf{CTL}$ and linear temporal logic $\mathbf{LTL}$. Both $\mathbf{CTL}$ and $\mathbf{LTL}$ are well-known temporal logics in computer science. First, we define the set of $\mathbf{CTL}$-formulas, $\mathrm{Fml}_{\mathbf{CTL}}$, as follows: $\mathrm{Atm}\subseteq\mathrm{Fml}_{\mathbf{CTL}}$ and for all $\varphi, \psi\in\mathrm{Fml}_{\mathbf{CTL}}$, $\varphi\rightarrow\psi\in\mathrm{Fml}_{\mathbf{CTL}}$, $\mathsf{AX}\varphi, \mathsf{EG}\varphi\in\mathrm{Fml}_{\mathbf{CTL}}$, and $\varphi\mathsf{EU}\psi\in\mathrm{Fml}_{\mathbf{CTL}}$.
\vspace{2mm}

We next formalize $\mathbf{CTL}$-models. Our formalization of the semantics over $\mathrm{RCA}_0$ is based on~\cite{MR1191162}.

\begin{dfn}[$\mathrm{RCA}_0$~\cite{MR1191162}]
A \emph{$\mathbf{CTL}$-model} is a tuple $M=(W, R, V)$ satisfying the following conditions:
\begin{enumerate}
    \item $(W, R)$ is a serial frame, that is, $\forall w\in W\exists u\in W(wRu)$,
    \item %$V$ is a function which assigns a truth value to each pair of a propositional formula and an element of $W$, i.e., 
$V : W\times\mathrm{Fml}_{\mathbf{CTL}} \rightarrow\{0, 1\}$,
    \item $V(w, \bot)=0$ for any $w\in W$, 
    \item $V(w, \varphi\rightarrow\psi)=1-V(w, \varphi)(1-V(w, \psi))$ for any $w\in W$ and any $\varphi, \psi\in\mathrm{Fml}_{\mathbf{CTL}}$, 
    \item $V(w, \mathsf{AX}\varphi)=1\iff\forall v\in W(wRv\rightarrow V(v, \varphi)=1)$ for any $w\in W$ and any $\varphi\in\mathrm{Fml}_{\mathbf{CTL}}$,
    \item $V(w, \mathsf{EG}\varphi)=1\iff\exists f: \mathbb{N}\rightarrow W\Bigl(f(0)=w\land\forall i, j\bigl(i<j\rightarrow f(i)Rf(j)\bigr)\\\land \bigl(\forall i\geq0 (V(f(i), \varphi)=1))\bigr)\Bigr)$ for any $w\in W$ and any $\varphi\in\mathrm{Fml}_{\mathbf{CTL}}$,
    \item $V(w, \varphi\mathsf{EU}\psi)=1\iff\exists f: \mathbb{N}\rightarrow W\Bigl(f(0)=w\land\forall i, j\bigl(i<j\rightarrow f(i)Rf(j)\bigr)\\\land \bigl(\exists i\geq0 \forall j<i(V(f(j), \varphi)=1\land V(f(i), \psi)=1))\bigr)\Bigr)$ for any $w\in W$ and any $\varphi, \psi\in\mathrm{Fml}_{\mathbf{CTL}}$.
\end{enumerate}

$V: W\times\mathrm{Fml}_{\mathbf{CTL}} \rightarrow\{0, 1\}$ is a \textit{$\mathbf{CTL}$-valuation on a serial frame $(W, R)$} if $(W, R, V)$ is a $\mathbf{CTL}$-model.
\end{dfn}

Next, we define the set of $\mathbf{LTL}$-formulas, $\mathrm{Fml}_{\mathbf{LTL}}$, as follows: $\mathrm{Atm}\subseteq\mathrm{Fml}_{\mathbf{LTL}}$ and for all $\varphi, \psi\in\mathrm{Fml}_{\mathbf{LTL}}$, $\varphi\rightarrow\psi\in\mathrm{Fml}_{\mathbf{LTL}}$, $\mathsf{X}\varphi\in\mathrm{Fml}_{\mathbf{LTL}}$, and $\varphi\mathsf{U}\psi\in\mathrm{Fml}_{\mathbf{LTL}}$.
\vspace{2mm}

We next formalize $\mathbf{LTL}$-models. There are several possible definitions of $\mathbf{LTL}$-models; see~\cite{MR2493187} for details. In this paper, we regard a frame as a single infinite path and consider $\mathbf{LTL}$-valuations on it.

\begin{dfn}[$\mathrm{RCA}_0$~\cite{MR2493187, MR1191162}]
A \emph{$\mathbf{LTL}$-model} is a tuple $M=(S, R, V)$ satisfying the following conditions:
\begin{enumerate}
    \item $S=\{s_{i}\}_{i\in\mathbb{N}}$ is an infinite path by $R$, that is, $\forall i, j\in\mathbb{N}\bigl(i<j\rightarrow s_{i}R\,s_{j}\bigr)$,
    \item %$V$ is a function which assigns a truth value to each pair of a propositional formula and an element of $W$, i.e., 
$V : S\times\mathrm{Fml}_{\mathbf{LTL}} \rightarrow\{0, 1\}$,
    \item $V(s_{i}, \bot)=0$ for any $i\in\mathbb{N}$, 
    \item $V(s_{i}, \varphi\rightarrow\psi)=1-V(s_{i}, \varphi)(1-V(s_{i}, \psi))$ for any $i\in\mathbb{N}$ and any $\varphi, \psi\in\mathrm{Fml}_{\mathbf{LTL}}$, 
    \item $V(s_{i}, \mathsf{X}\varphi)=1\iff V(s_{i+1}, \varphi)=1$ for any $i\in\mathbb{N}$ and any $\varphi\in\mathrm{Fml}_{\mathbf{LTL}}$,
    \item $V(s_{i}, \varphi\mathsf{U}\psi)=1\iff\exists n\geq i\bigl(\forall j(i\leq j<n\rightarrow(V(s_{j}, \varphi)=1))\land V(s_{n}, \psi)=1\bigr)$ for any $i\in\mathbb{N}$ and any $\varphi, \psi\in\mathrm{Fml}_{\mathbf{LTL}}$.
\end{enumerate}

We call the above pair $(S, R)$ a \emph{$\mathbf{LTL}$-frame}. A function $V: S\times\mathrm{Fml}_{\mathbf{LTL}} \rightarrow\{0, 1\}$ is a \textit{$\mathbf{LTL}$-valuation on a $\mathbf{LTL}$-frame $(S, R)$} if $(S, R, V)$ is a $\mathbf{LTL}$-model.
\end{dfn}

\subsection{The Valuation Extension Lemma for $\mathbf{CTL}$}
We now consider the logical strength of the following three statements.
\vspace{2mm}

\begin{state}[The Restricted Valuation Extension Lemma for $\mathbf{CTL}$]
For each serial frame $F=(W, R)$, $n\in\mathbb{N}$, and $v:W\times\mathrm{Prop}\rightarrow2$, there exists a $n$-restricted $\mathbf{CTL}$-valuation $V$ on $(W, R)$ such that $V$ is an extension of $v$.
\end{state}

\begin{state}[The Valuation Extension Lemma for $\mathbf{CTL}$]
For each serial frame $F=(W, R)$ and $v:W\times\mathrm{Prop}\rightarrow2$, there exists a $\mathbf{CTL}$-valuation $V$ on $(W, R)$ such that $V$ is an extension of $v$.
\end{state}

\begin{state}[The Valuation Extension Lemma for image-finite $\mathbf{CTL}$]
For each image-finite serial frame $F=(W, R)$ and $v:W\times\mathrm{Prop}\rightarrow2$, there exists a $\mathbf{CTL}$-valuation $V$ on $(W, R)$ such that $V$ is an extension of $v$.
\end{state}
\vspace{2mm}

First, we consider Restricted VEL for $\mathbf{CTL}$. By the preceding discussion, this statement implies $\mathrm{ACA}_0$ over $\mathrm{RCA}_0$.
\begin{prop}[$\mathrm{RCA}_0$]
Restricted VEL for $\mathbf{CTL}$ implies $\mathrm{ACA}_{0}$.
\end{prop}

\begin{proof}
It suffices to use the frame constructed in Proposition~\ref{velraca}.
\end{proof}

\begin{cor}[$\mathrm{RCA}_0$]
VEL for image-finite $\mathbf{CTL}$ implies $\mathrm{ACA}_{0}$.
\end{cor}

Moreover, Restricted VEL for $\mathbf{CTL}$ implies $\Pi^1_1$-comprehension. We use the following lemma for the proof.

\begin{lem}[\cite{MR2517689}]\label{pi11catree}
The following are equivalent over $\mathrm{RCA}_0$:
\begin{enumerate}
    \item $\Pi^1_1$-$\mathrm{CA}_0$,
    \item For any sequence of trees $\langle T_{k}\mid k\in\mathbb{N}\rangle$, $T_{k}\subseteq\mathbb{N}^{<\mathbb{N}}$, there exists a set $X$ such that $\forall k(k\in X\leftrightarrow T_{k}\text{ has a path})$.
\end{enumerate}
\end{lem}

\begin{thm}[$\mathrm{RCA}_0$]
Restricted VEL for $\mathbf{CTL}$ implies $\Pi^1_1$-$\mathrm{CA}_0$.
\end{thm}

\begin{proof}
By Lemma~\ref{pi11catree}, it is enough to show that for any sequence of trees $\langle T_{k}\mid k\in\mathbb{N}\rangle$, $T_{k}\subseteq\mathbb{N}^{<\mathbb{N}}$, there exists a set $X$ such that $\forall k(k\in X\leftrightarrow T_{k}\text{ has a path})$. 

Let $\langle T_{k}\subseteq\mathbb{N}^{<\mathbb{N}}\mid k\in\mathbb{N}\rangle$ be a sequence of trees, $W=\mathbb{N}\times\mathbb{N}^{<\mathbb{N}}$, and let $\varepsilon$ be the empty sequence in $\mathbb{N}^{<\mathbb{N}}$. We define $R\subseteq W\times W$ as follows: for all $\langle{k, \sigma\rangle}, \langle{k^{\prime}, \sigma^{\prime}\rangle}\in W$,
\begin{align*}
\langle{k, \sigma\rangle}R\langle{k^{\prime}, \sigma^{\prime}\rangle}\Leftrightarrow (k=k^{\prime})\land(\sigma\prec\sigma^{\prime})\land(lh(\sigma)+1=lh(\sigma^{\prime})).
\end{align*}
Then $(W, R)$ is a serial frame. Let $p_{0}\in\mathrm{Prop}$. We define $v: W\times\mathrm{Prop}\rightarrow 2$ as follows: for all $\langle{k, \sigma\rangle}\in W$,
\begin{align*}
&v(\langle{k, \sigma\rangle}, p_{0})=1\Leftrightarrow\sigma\in T_{k},\\
&v(\langle{k, \sigma\rangle}, p)=1,\text{ for all $p\in\mathrm{Prop}\setminus\{p_{0}\}$}.
\end{align*}

By Restricted VEL for $\mathbf{CTL}$, there exists a $1$-restricted $\mathbf{CTL}$-valuation $V$ on $(W, R)$ such that $V$ is an extension of $v$. We define $X=\{k\in\mathbb{N}\mid V(\langle{k, \varepsilon\rangle}, \mathsf{EG}p_{0})=1\}$. This $X$ is the desired set. We show that for all $k\in\mathbb{N}$, $k\in X\Leftrightarrow\text{$T_{k}$ has a path.}$ Fix $k\in\mathbb{N}$.
\vspace{2mm}

$(\Rightarrow):$ Assume that $k\in X$. By $V(\langle{k, \varepsilon\rangle}, \mathsf{EG}p_{0})=1$, there are paths, $f: \mathbb{N}\rightarrow W$, and  $g: \mathbb{N}\rightarrow\mathbb{N}^{<\mathbb{N}}$ such that for all $i\geq0$, $f(i)=\langle{k, g(i)\rangle}$, $lh(g(i))=i$, and $g(i)\in T_{k}$. Then we define $h: \mathbb{N}\rightarrow T_{k}$; $h(n)=g(n)$. Thus, $h$ is a path of $T_{k}$. 
\vspace{1mm}

$(\Leftarrow):$ Assume that $k\not\in X$. By $V(\langle{k, \varepsilon\rangle}, \mathsf{EG}p_{0})=0$, for paths, $f: \mathbb{N}\rightarrow W$ and $g: \mathbb{N}\rightarrow\mathbb{N}^{<\mathbb{N}}$, there exists $i\geq0$ such that if $f(i)=\langle{k, g(i)\rangle}$ and $lh(g(i))=i$, then $g(i)\not\in T_{k}$. To derive a contradiction, assume that $T_{k}$ has a path $h$. Then we define $F: \mathbb{N}\rightarrow W$; $F(i)=\langle{k, h(i)\rangle}$. Then $F$ is a path of $W$ and $h(i)\in T_{k}$. It is a contradiction by assumption. Thus, $T_{k}$ has no path. 
\end{proof}

Conversely, Restricted VEL for $\mathbf{CTL}$ is provable in $\Pi^1_1\text{-}\mathrm{CA}_0\, + \,\Sigma^1_2\text{-}\text{induction}$. It follows that VEL for $\mathbf{CTL}$ is provable in $\Sigma^1_2\text{-}\mathrm{DC}_0$.

\begin{thm}[The Restricted Valuation Extension Lemma for $\mathbf{CTL}$, $\Pi^1_1\text{-}\mathrm{CA}_0\, + \,\Sigma^1_2\text{-}\text{induction}$ ]\label{rvelctl}
For each serial frame $F=(W, R)$, $n\in\mathbb{N}$, and $v:W\times\mathrm{Prop}\rightarrow2$, there exists a $n$-restricted $\mathbf{CTL}$-valuation $V$ on $(W, R)$ such that $V$ is an extension of $v$.
\end{thm}

\begin{proof}
Let $F=(W, R)$ be a serial frame and $v:W\times\mathrm{Prop}\rightarrow2$. Suppose $\mathrm{Eval}(n, V)$ is the formula stating that $V$ is an $n$-restricted valuation on $F$ extending $v$. Then, for each number $n$ and set $V$, $\mathrm{Eval}(n, V)$ is a $\Sigma^1_2$ sentence. This follows by writing out the definition of a $\mathbf{CTL}$-model using $\Sigma^1_1\text{-}\mathrm{AC}_0$ (that is, $\Pi^1_1\text{-}\mathrm{CA}_0$). Now, we consider proving that for each $n$, $\exists V \mathrm{Eval}(n, V)$ holds using $\Pi^1_1\text{-}\mathrm{CA}_0$ plus $\Sigma^1_2$-induction.
\vspace{1mm}

Step $0$: We define $V_{0}:W\times\mathrm{Fml}_{\mathbf{CTL}}[0]\rightarrow2$ as follows: for all $w\in W$ and $p\in\mathrm{Prop}$, $V_0(w, p)=v(w, p)$, and for all $w\in W$ and $\varphi, \psi\in\mathrm{Fml}_{\mathbf{CTL}}[0]$ for which the values of $V_{0}$ have already been defined, $V_{0}(w, \varphi\rightarrow\psi)=1-V_0(w, \varphi)(1-V_0(w, \psi))$. Then $V_{0}$ is a $0$-restricted valuation on $F$. Thus, $\exists V\mathrm{Eval}(0, V)$ holds.  
\vspace{1mm}

Step $k$: Assume that $\exists V\mathrm{Eval}(k-1, V)$ holds. Choose a $V_{k-1}$ such that $\mathrm{Eval}(k-1, V_{k-1})$ holds. We define the following sets.
\begin{align*}
AX_{k-1}=\Bigl\{\langle{w, \varphi\rangle}\in W\times\mathrm{Fml}_{\mathbf{CTL}}[k-1]\mid\forall v\in W(wRv\rightarrow V_{k-1}(v, \varphi)=1)\Bigr\},
\end{align*}
\begin{align*}
EG_{k-1}=\Bigl\{\langle{w, \varphi\rangle}\in W\times\mathrm{Fml}_{\mathbf{CTL}}[k-1]\mid \exists f: \mathbb{N}\rightarrow W\Bigl(f(0)=w\land\forall i, j\bigl(i<j\rightarrow f(i)Rf(j)\bigr)\\\land \bigl(\forall i\geq0 (V_{k-1}(f(i), \varphi)=1))\bigr)\Bigr)\Bigr\},
\end{align*}
\begin{align*}
EU_{k-1}=\Bigl\{\langle{w, \varphi, \psi\rangle}\in W\times\mathrm{Fml}_{\mathbf{CTL}}[k-1]\times\mathrm{Fml}_{\mathbf{CTL}}[k-1]\mid\exists f: \mathbb{N}\rightarrow W\Bigl(f(0)=w\land\forall i, j\bigl(i<j\rightarrow f(i)Rf(j)\bigr)\\\land \bigl(\exists i\geq0 \forall j<i(V_{k-1}(f(j), \varphi)=1\land V_{k-1}(f(i), \psi)=1))\bigr)\Bigr)\Bigr\}.
\end{align*}

By $\Pi^1_1\text{-}\mathrm{CA}_0$, such sets $AX_{k-1}$, $EG_{k-1}$, and $EU_{k-1}$ exist. Then we define $V_{k} : W\times\mathrm{Fml}_{\mathbf{CTL}}[k]\rightarrow 2$ as follows: fix $w\in W$ and $\varphi\in\mathrm{Fml}_{\mathbf{CTL}}[k]$. 

If $\mathrm{deg}(\varphi)\leq k-1$, then $V_{k}(w, \varphi)\coloneqq V_{k-1}(w, \varphi)$. Assume that $\mathrm{deg}(\varphi)=k$. We define the value of $V_{k}$ by the complexity of $\varphi$.
\begin{itemize}
    \item Let $\varphi\equiv\mathsf{AX}\psi$. If $\langle{w, \psi\rangle}\in AX_{k-1}$, then $V_{k}(w, \varphi)\coloneqq 1$, otherwise, $V_{k}(w, \varphi)\coloneqq 0$,
    \item Let $\varphi\equiv\mathsf{EG}\psi$. If $\langle{w, \psi\rangle}\in EG_{k-1}$, then $V_{k}(w, \varphi)\coloneqq 1$, otherwise, $V_{k}(w, \varphi)\coloneqq 0$,
    \item Let $\varphi\equiv\psi\mathsf{EU}\theta$. If $\langle{w, \psi, \theta\rangle}\in EU_{k-1}$, then $V_{k}(w, \varphi)\coloneqq 1$, otherwise, $V_{k}(w, \varphi)\coloneqq 0$.
\end{itemize}
For $w\in W$ and $\varphi, \psi\in\mathrm{Fml}_{\mathbf{CTL}}[k]$ for which the values of $V_{k}$ have already been defined, we set
\begin{align*}
V_{k}(w, \varphi\rightarrow\psi)\coloneqq1-V_{k}(w, \varphi)(1-V_{k}(w, \psi)).
\end{align*}
Then this $V_{k}$ is a $k$-restricted valuation on $F$ extending $v$. Thus, $\exists V\mathrm{Eval}(k, V)$ holds. 
\vspace{2mm}

By $\Sigma^1_2$-induction, for all $n$, $\exists V\mathrm{Eval}(n, V)$ holds. Thus, Restricted VEL for $\mathbf{CTL}$ holds.
\end{proof}

%\begin{cor}
%The Restricted Valuation Extension Lemma for $\mathbf{CTL}$-frames does not imply $\Sigma^1_2\text{-$\mathrm{AC}_0$}$ over $\mathrm{ACA}_0$.
%\end{cor}

\begin{cor}[The Valuation Extension Lemma for $\mathbf{CTL}$, $\Sigma^1_2$-$\mathrm{DC}_0$]
For each serial frame $F=(W, R)$ and $v:W\times\mathrm{Prop}\rightarrow2$, there exists a $\mathbf{CTL}$-valuation $V$ on $(W, R)$ such that $V$ is an extension of $v$.
\end{cor}

\begin{proof}
Let $F=(W, R)$ be a serial frame and $v:W\times\mathrm{Prop}\rightarrow2$. By Theorem~\ref{rvelctl}, $\Pi^1_1\text{-}\mathrm{CA}_0 + \Sigma^1_2\text{-}\text{induction}$ implies $\forall n\exists V\mathrm{Eval}(n, V)$. Thus, by $\Sigma^1_2\text{-}\mathrm{AC}_0$, $\exists V\forall n\mathrm{Eval}(n, V_{n})$ holds, which yields VEL for $\mathbf{CTL}$. Since $\Sigma^1_2\text{-}\mathrm{AC}_0\, +\, \Sigma^1_2\text{-}\text{induction}$ is equivalent to $\Sigma^1_2\text{-}\mathrm{DC}_0$, we conclude that $\Sigma^1_2\text{-}\mathrm{DC}_0$ implies VEL for $\mathbf{CTL}$.
\end{proof}

Next, we analyze the logical strength of VEL for image-finite $\mathbf{CTL}$ from the viewpoint of computability theory. 

\begin{prop}\label{ac2bctlvel}
Let $A$ be a set and $k\in\omega$. Then there exist a computable $2$-branching frame $F=(W, R)$ and a $A$-computable $v : W\times\mathrm{Prop}\rightarrow 2$ such that for any valuation $V:W\times \mathrm{Fml}_{\mathbf{CTL}}\rightarrow 2$ on $(W, R)$ extending $v$, $V$ computes $A^{(2k+1)}$.
\end{prop}

\begin{proof}
Let $A$ be a set and $k\in\omega$. Let $\{\exists n_{1}\forall n_{2}\exists n_{3}\dots\forall n_{2k}\exists n_{2k+1}\theta_{i}(\vec{n}, x, X)\}_{i\in\omega}$ be a computable listing of $\Sigma^{0}_{2k+1}$-formulas with no free variables except a number variable $x$ and a set variable $X$. (The $\theta_{i}$'s are bounded formulas).  Let $W=\omega^{<\omega}$. Then we define $R\subseteq W\times W$ as follows: for $\sigma, \tau\in W$, 
\begin{align*}
\sigma R\tau\Leftrightarrow&\Big(\sigma\preceq\tau\land(lh(\tau)=lh(\sigma)+1)\land(\tau_{lh(\tau)-1}=0)\Bigr)\\&\lor\Bigl((lh(\sigma)=lh(\tau)=n)\land(\sigma\upharpoonright(n-1)=\tau\upharpoonright(n-1))\land(\tau_{lh(\tau)-1}=\sigma_{lh(\sigma)-1}+1)\Bigr).
\end{align*} 

This gives the frame $(W, R)$. The following diagram illustrates the part of $(W, R)$ reachable from $\langle{n, e, 0\rangle}$. The transition relation $R$ extends from each node to the node immediately above it and to the node immediately to its right. 
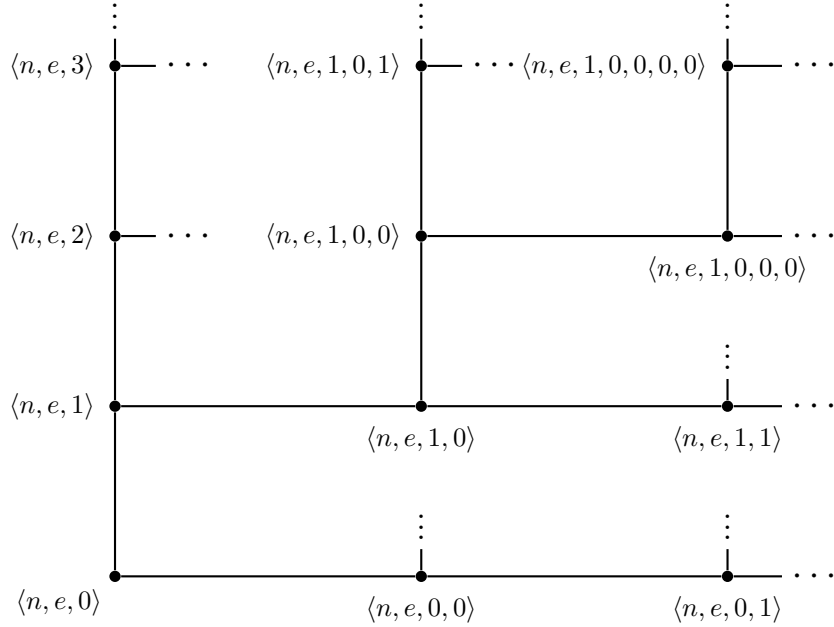
\begin{center}
\begin{tikzpicture}[
    scale=0.9,
    node/.style={circle, fill=black, inner sep=1.5pt}, % Node size
    dots/.style={font=\Large} % Slightly larger ellipsis symbols
]

% ====================
% Nodes & Labels
% ====================

% Left column (x=0)
\node[node] (n00) at (0, 0) {};
\node[below left=2pt] at (n00) {$\langle n, e, 0 \rangle$};

\node[node] (n01) at (0, 2.5) {};
\node[left=4pt] at (n01) {$\langle n, e, 1 \rangle$};

\node[node] (n02) at (0, 5) {};
\node[left=4pt] at (n02) {$\langle n, e, 2 \rangle$};

\node[node] (n03) at (0, 7.5) {};
\node[left=4pt] at (n03) {$\langle n, e, 3 \rangle$};

% Middle column (x=4.5)
\node[node] (n10) at (4.5, 0) {};
\node[below=4pt] at (n10) {$\langle n, e, 0, 0 \rangle$};

\node[node] (n11) at (4.5, 2.5) {};
\node[below=4pt] at (n11) {$\langle n, e, 1, 0 \rangle$}; % Moved directly below

\node[node] (n12) at (4.5, 5) {};
\node[left=4pt] at (n12) {$\langle n, e, 1, 0, 0 \rangle$};

\node[node] (n13) at (4.5, 7.5) {};
\node[left=4pt] at (n13) {$\langle n, e, 1, 0, 1 \rangle$};

% Right column (x=9)
\node[node] (n20) at (9, 0) {};
\node[below=4pt] at (n20) {$\langle n, e, 0, 1 \rangle$};

\node[node] (n21) at (9, 2.5) {};
\node[below=4pt] at (n21) {$\langle n, e, 1, 1 \rangle$}; % Moved directly below

\node[node] (n22) at (9, 5) {};
\node[below=4pt] at (n22) {$\langle n, e, 1, 0, 0, 0 \rangle$};

\node[node] (n23) at (9, 7.5) {}; 
\node[left=4pt] at (n23) {$\langle n, e, 1, 0, 0, 0, 0 \rangle$};

% ====================
% Solid Lines
% ====================

% Vertical lines (solid parts)
\draw[thick] (n00) -- (n01) -- (n02) -- (n03);
\draw[thick] (n11) -- (n12) -- (n13);
\draw[thick] (n22) -- (n23);

% Horizontal lines (solid parts)
\draw[thick] (n00) -- (n10) -- (n20);
\draw[thick] (n01) -- (n11) -- (n21);
\draw[thick] (n12) -- (n22); % Reconnected

% ====================
% Continuation Dots (... and \vdots)
% ====================

% Horizontal continuations (short lines extending to the right and ...)
\draw[thick] (n02) -- ++(0.6, 0) node[right, dots] {$\cdots$};
\draw[thick] (n03) -- ++(0.6, 0) node[right, dots] {$\cdots$};

% Since n12 is connected, remove its continuation and keep only n13
\draw[thick] (n13) -- ++(0.6, 0) node[right, dots] {$\cdots$};

\draw[thick] (n20) -- ++(0.8, 0) node[right, dots] {$\cdots$};
\draw[thick] (n21) -- ++(0.8, 0) node[right, dots] {$\cdots$};
\draw[thick] (n22) -- ++(0.8, 0) node[right, dots] {$\cdots$};
\draw[thick] (n23) -- ++(0.8, 0) node[right, dots] {$\cdots$};

% Vertical continuations (short upward lines and vdots)
\draw[thick] (n03) -- ++(0, 0.4) node[above, dots] {$\vdots$};
\draw[thick] (n13) -- ++(0, 0.4) node[above, dots] {$\vdots$};
\draw[thick] (n23) -- ++(0, 0.4) node[above, dots] {$\vdots$};

% Add short upward lines and vdots from n10, n20, and n21
\draw[thick] (n10) -- ++(0, 0.4) node[above, dots] {$\vdots$};
\draw[thick] (n20) -- ++(0, 0.4) node[above, dots] {$\vdots$};
\draw[thick] (n21) -- ++(0, 0.4) node[above, dots] {$\vdots$};

\end{tikzpicture}
\end{center}
Fix $q, r\in\mathrm{Prop}$. Then we define $v:W\times\mathrm{Prop}\rightarrow2$ by the following: for all $\sigma\in W$,
\begin{align*}
&v(\sigma, q)=1\Leftrightarrow lh(\sigma)\text{ is odd},\\
&v(\sigma, r)=1\Leftrightarrow lh(\sigma)=2k+3\land\theta_{\sigma_{1}}(\sigma_{2}, \sigma_{3}, \dots, \sigma_{lh(\sigma)-1}, \sigma_{0}, A),\\
&v(\sigma, p)=1,\, \text{for all $p\in\mathrm{Prop}\setminus\{q, r\}$}.
\end{align*}

Fix a valuation $V:W\times \mathrm{Fml}_{\mathbf{CTL}}\rightarrow 2$ on $(W, R)$ extending $v$. We show that $A^{(2k+1)}\leq_{T}V$. We take a $\Sigma^{0}_{2k+1}$-formula $\varphi(x, A)$ such that for all $n$, $\varphi(n, A)\Leftrightarrow n\in A^{(2k+1)}$. Then, there exists $e$ such that $\varphi(x, A)\equiv\exists n_{1}\forall n_{2}\exists n_{3}\dots\forall n_{2k}\exists n_{2k+1}\theta_{e}(\vec{n}, x, A)$. 
\vspace{2mm}

We consider a specific nested $\mathbf{CTL}$-formula $\Psi_{2k+1}$ defined by
\[
\Psi_{2k+1} = q \textsf{EU} \bigl( q \land \textsf{AX} ( \neg q \rightarrow \textsf{EG} ( \neg q \land \textsf{AX} ( q \rightarrow \dots q \textsf{EU} ( q \land r ) \dots ) ) ) \bigr),
\]
where the nesting of the modal operators $q \textsf{EU} \bigl( q \land \textsf{AX} ( \neg q \rightarrow \textsf{EG} ( \neg q \land \textsf{AX} ( q \rightarrow \dots \bigr)$ is repeated exactly $k$ times, corresponding to the $k$ pairs of alternating quantifiers in $\varphi$. 

\begin{lem}\label{PSIlem}
Let $\psi\in\mathrm{Fml}_{\mathbf{CTL}}$, $\sigma\in W$, and let $\Psi \equiv q \mathsf{EU} \bigl( q \land \mathsf{AX} ( \neg q \rightarrow \mathsf{EG} ( \neg q \land \mathsf{AX} ( q \rightarrow \psi ) ) ) \bigr)$. Suppose that $lh(\sigma)\geq3$, $lh(\sigma)$ is odd, and $\sigma_{lh(\sigma)-1}=0$. Then the following holds.
\begin{align*}
V(\sigma, \Psi)=1\Leftrightarrow\exists m\forall s\Bigl(V(\langle{\sigma_{0}, \sigma_{1}, \dots, \sigma_{lh(\sigma)-2}, m, s, 0\rangle}, \psi)=1\Bigr).
\end{align*} 
\end{lem}

\begin{lemproof}
$(\Rightarrow)$: We assume that $V(\sigma, \Psi)=1$. Then there exist a path $f: \mathbb{N}\rightarrow W$ and $i\geq0$ such that $f(0)=\sigma$ and for all $j\leq i$, $V(f(j), q)=1$, and $V\bigl(f(i), \mathsf{AX} ( \neg q \rightarrow \mathsf{EG} ( \neg q \land \mathsf{AX} ( q \rightarrow \psi ) ) )\bigr)=1$. Take such an $f$. Then for all $j\leq i$, $lh(f(j))=lh(f(i))=lh(\sigma)$. Let $\tau=f(i)$ and $m=\tau_{lh(\tau)-1}$. Fix $s\geq0$. Note that for all $j\leq lh(\sigma)-2$, $\tau_{j}=\sigma_{j}$. Since $\tau R\langle{\tau_{0},\dots,\tau_{lh(\tau)-1}, 0\rangle}$, $V\bigl(\langle{\tau_{0},\dots,\tau_{lh(\tau)-1}, 0\rangle},    \mathsf{EG} ( \neg q \land \mathsf{AX} ( q \rightarrow \psi ) ) \bigr)=1$ holds. Then there exists a path $g: \mathbb{N}\rightarrow W$ such that $g(0)=\langle{\tau_{0},\dots,\tau_{lh(\tau)-1}, 0\rangle}$ and for all $j\geq0$, $V(g(j), \neg q)=1$ and $V\bigl(g(j), \mathsf{AX} ( q \rightarrow \psi )  \bigr)=1$ hold. Take such a $g$. Then for all $j\geq0$, $lh(g(j))=lh(g(0))$. Let $g(s)=\langle{\tau_{0},\dots,\tau_{lh(\tau)-1}, s\rangle}$. Then $V(\langle{\tau_{0},\dots,\tau_{lh(\tau)-1}, s, 0\rangle}, \psi)=1$ holds. Thus, $\exists m\forall s\bigl(V(\langle{\sigma_{0}, \sigma_{1}, \dots, \sigma_{lh(\sigma)-2}, m, s, 0\rangle}, \psi)=1\bigr)$ holds.
\vspace{2mm}

$(\Leftarrow)$: Assume that $V(\sigma, \Psi)=0$. Fix $m$. By assumption, for any path $f: \mathbb{N}\rightarrow W$ with $f(0)=\sigma$ and $i\geq0$, there exists $j\leq i$ such that $V(f(j), q)=0$ or $V\bigl(f(i), \mathsf{AX} ( \neg q \rightarrow \mathsf{EG} ( \neg q \land \mathsf{AX} ( q \rightarrow \psi ) ) )\bigr)=0$. We define a path $g$ as follows: for all $k$, $g(k)=\langle{\sigma_{0}, \sigma_{1}, \dots, \sigma_{lh(\sigma)-2}, k\rangle}$. Then, by assumption, $V\bigl(g(m), \mathsf{AX} ( \neg q \rightarrow \mathsf{EG} ( \neg q \land \mathsf{AX} ( q \rightarrow \psi ) ) )\bigr)=0$. Thus, $V\bigl(\langle{\sigma_{0}, \sigma_{1}, \dots, \sigma_{lh(\sigma)-2}, m, 0\rangle}, \mathsf{EG} ( \neg q \land \mathsf{AX} ( q \rightarrow \psi ) ) \bigr)=0$ holds. Then for any path $f: \mathbb{N}\rightarrow W$ with $f(0)=\langle{\sigma_{0}, \sigma_{1}, \dots, \sigma_{lh(\sigma)-2}, m, 0\rangle}$, there exists $s\geq0$ such that $V(f(s), \neg q)=0$ or $V\bigl(f(s), \mathsf{AX} ( q \rightarrow \psi )  \bigr)=0$ hold. We define a path $h$ as follows: for all $k$, $h(k)=\langle{\sigma_{0}, \sigma_{1}, \dots, \sigma_{lh(\sigma)-2}, m, k\rangle}$. Then, by assumption, there exists $s\geq0$ such that $V\bigl(h(s), \mathsf{AX} ( q \rightarrow \psi )  \bigr)=0$, that is, $V\bigl(\langle{\sigma_{0}, \sigma_{1}, \dots, \sigma_{lh(\sigma)-2}, m, s\rangle}, \mathsf{AX} ( q \rightarrow \psi )  \bigr)=0$. Thus, $\forall m\exists s\bigl(V(\langle{\sigma_{0}, \sigma_{1}, \dots, \sigma_{lh(\sigma)-2}, m, s, 0\rangle}, \psi)=0\bigr)$ holds. 
\end{lemproof}

Fix $\langle n, e, 0 \rangle \in W$. Then, by Lemma~\ref{PSIlem}, the chain of equivalences holds as follows:
\begin{align*}
V(\langle n, e, 0 \rangle, \Psi_{2k+1}) = 1
&\Leftrightarrow \exists n_{1} \forall n_{2} \dots \exists n_{2k-1}\forall n_{2k} \, V\bigl(\langle n, e, n_{1}, n_{2}, \dots ,n_{2k}, 0 \rangle, q \textsf{EU} ( q \land r )\bigr)=1 \\
&\Leftrightarrow \exists n_{1} \forall n_{2} \dots \exists n_{2k-1}\forall n_{2k} \exists n_{2k+1} \, V\bigl(\langle n, e, n_{1}, n_{2}, \dots ,n_{2k}, n_{2k+1} \rangle, r \bigr)=1 \\
&\Leftrightarrow \exists n_{1} \forall n_{2} \dots \exists n_{2k-1}\forall n_{2k} \exists n_{2k+1} \, \theta_{e}(n_{1}, \dots, n_{2k+1}, n, A) \\
&\Leftrightarrow \varphi(n, A) \\
&\Leftrightarrow n \in A^{(2k+1)}.
\end{align*}
Thus, $V$ computes $A^{(2k+1)}$.
\end{proof}

\begin{cor}\label{ctla2branch}
Let $A$ be a set. Then there exist a computable $2$-branching frame $F=(W, R)$ and a $A$-computable $v : W\times\mathrm{Prop}\rightarrow 2$ such that for any valuation $V:W\times \mathrm{Fml}_{\mathbf{CTL}}\rightarrow 2$ on $(W, R)$ extending $v$, $V$ computes $A^{(\omega)}$.
\end{cor}

\begin{proof}
This follows by an argument similar to that of Proposition~\ref{ac2bctlvel}.
\end{proof}

In fact, VEL for image-finite $\mathbf{CTL}$ is equivalent to $\mathrm{ACA}^{+}_0$ over $\mathrm{RCA}_0$.

\begin{thm}
The following statements are equivalent over $\mathrm{RCA}_0$:
\begin{enumerate}
    \item $\mathrm{ACA}^{+}_0$,
    \item The Valuation Extension Lemma for image-finite $\mathbf{CTL}$.
\end{enumerate}
\end{thm}

\begin{proof}
$(2)\Rightarrow(1)$: By Corollary~\ref{ctla2branch}.
\vspace{2mm}

$(1)\Rightarrow(2)$: Let $F=(W, R)$ be an image-finite serial frame and $v: W\times\mathrm{Prop}\rightarrow 2$. Then, we construct a valuation $V: W \times \mathrm{Fml}_{\mathbf{CTL}} \to 2$ extending $v$ in a bottom-up manner, using $(F\oplus v)^{(\omega)}$ as an oracle. The construction is similar to that in Theorem~\ref{pextend}. Step $0$ is carried out in the same way as in Theorem~\ref{pextend}; we describe Step $k$.
\vspace{2mm}

Step $k$: We assume that $V_{k-1}: W \times \mathrm{Fml}_{\mathbf{CTL}}[k-1] \to 2$ is defined and $e_{k-1}$ is an index of $V_{k-1}$. Fix $w\in W$ and $\varphi,\psi\in\mathrm{Fml}_{\mathbf{CTL}}[k-1]$. Then, using $(F \oplus v)^{(\omega)}$ as an oracle, we determine the value by answering the following queries.
\vspace{3mm}

$(\mathrm{A})$ $\forall u \in W (wRu \rightarrow V_{k-1}(u, \varphi)=1)$ holds ? If yes, then $V_{k}(w, \mathsf{AX}\varphi)=1$, otherwise, $V_{k}(w, \mathsf{AX}\varphi)=0$. 
\vspace{1mm}

$(\mathrm{B})$ $\forall n\geq1\exists \sigma\in\mathrm{Sec}^{=n}(W)\Bigl((\sigma_{0}=w)\land\forall i< n(V_{k-1}(\sigma_{i}, \varphi)=1)\Bigr)$ holds ? If yes, then $V_{k}(w, \mathsf{EG}\varphi)=1$, otherwise, $V_{k}(w, \mathsf{EG}\varphi)=0$. 
\vspace{1mm}

$(\mathrm{C})$ $\exists n\geq1\exists \sigma\in\mathrm{Sec}^{=n}(W)\Bigl((\sigma_{0}=w)\land(\forall i<n(V_{k-1}(\sigma_{i}, \varphi)=1))\land(V_{k-1}(\sigma_{n-1}, \psi)=1)\Bigr)$ holds ? If yes, then $V_{k}(w, \varphi\mathsf{EU}\psi)=1$, otherwise, $V_{k}(w, \varphi\mathsf{EU}\psi)=0$. 
\vspace{3mm}

Each of the queries $(\mathrm{A})$, $(\mathrm{B})$, and $(\mathrm{C})$ is arithmetical with $(F \oplus v)^{(k-1)}$ as a parameter. Therefore, whether they hold can be decided by using $(F \oplus v)^{(\omega)}$. Using $(F \oplus v)^{(\omega)}$ as an oracle, we compute an index $e_{k}$ for the $k$-restricted valuation $V_{k} : W \times \mathrm{Fml}_{\mathbf{CTL}}[k] \to 2$ defined in this way. 
\vspace{1mm}

This construction can be carried out in $\mathrm{ACA}_0$ using $(F\oplus v)^{(\omega)}$ as an oracle. In this construction, we have the indices $E=\{e_{k}\}_{k}$. The rest is the same as in Theorem~\ref{pextend}.
\vspace{2mm}

Next, we show that the function $V$ obtained by this construction is indeed a valuation on $(W, R)$. To this end, it suffices to verify that the value of $\mathsf{EG}\varphi$ determined in $(\mathrm{B})$ satisfies the definition of a valuation. That is, fixing any $w\in W$ and $\varphi\in\mathrm{Fml}_{\mathbf{CTL}}$, it is enough to show the equivalence of the following conditions:
\begin{enumerate}
    \item $\forall n\geq1\exists \sigma\in\mathrm{Sec}^{=n}(W)(\sigma_{0}=w\land\forall i< n(V_{k}(\sigma_{i}, \varphi)=1))$,
    \item  $\exists f: \mathbb{N}\rightarrow W\Bigl(f(0)=w\land\forall i, j\bigl(i<j\rightarrow f(i)Rf(j)\bigr)\land \bigl(\forall i\geq0 (V(f(i), \varphi)=1))\bigr)\Bigr)$.
\end{enumerate}

It is clear that $(2)$ implies $(1)$. We show that $(1)$ implies $(2)$. Consider a tree rooted at $w$ whose branches are formed by transitions to states satisfying $\varphi$. Since the frame is image-finite, this tree is finitely branching. Moreover, by $(1)$, this tree is infinite. Therefore, by K\"onig's Lemma (that is, $\mathrm{ACA}_0$), we obtain an infinite path $f$ starting from $w$. Thus, $(2)$ holds.
\end{proof}

In fact, if we do not require the valuation to extend a given assignment of propositional variables, it is easy to define a $\mathbf{CTL}$-valuation on a serial frame.

\begin{thm}[$\mathrm{RCA}_0$]
For each serial frame $F=(W, R)$, there exists a $\mathbf{CTL}$-valuation $V$ on $(W, R)$.
\end{thm}

\begin{proof}
Let $F$ be a serial frame. We define $v: W\times\mathrm{Prop}\rightarrow 2$ as follows: for all $w\in W$ and $p\in\mathrm{Prop}$, $v(w, p)=1$; that is, $v$ is a trivial assignment for propositions. Now, consider a frame consisting of a singleton set with a self-loop, denoted by $F_{a}=(W_{a}, R_{a})=(\{a\}, \{\langle{a, a\rangle}\})$. On this frame, let the proposition valuation $v_{a}: W_{a}\times\mathrm{Prop}\rightarrow 2$ also be the trivial one. Then, the valuation $V_{a}$ on $F_{a}$ can be constructed recursively. We now define $V:W\times\mathrm{Fml}_{\mathbf{CTL}}\rightarrow 2$ as follows: for each state $w\in W$ of $F$ and each $\varphi\in\mathrm{Fml}_{\mathbf{CTL}}$, let $V(w, \varphi)\coloneqq V_{a}(a, \varphi)$. Then, $V$ is indeed a $\mathbf{CTL}$-valuation on $F$.
\end{proof}

\subsection{The Valuation Extension Lemma for $\mathbf{LTL}$}

Next, we consider VEL for $\mathbf{LTL}$. 
\vspace{2mm}

\begin{state}[The Valuation Extension Lemma for $\mathbf{LTL}$]
For each $\mathbf{LTL}$-frame $F=(S, R)$ and $v:S\times\mathrm{Prop}\rightarrow2$, there exists a $\mathbf{LTL}$-valuation $V$ on $(S, R)$ such that $V$ is an extension of $v$.
\end{state}
\vspace{2mm}

First, we analyze it from the viewpoint of computability theory. The preceding arguments show that, for some computable frame and computable assignment of propositional variables on it, any extending valuation computes $\emptyset^{\prime}$. Moreover, one can construct a computable frame and a computable assignment of propositional variables on it such that any extending valuation computes $\emptyset^{\prime\prime}$.

\begin{prop}\label{ltla2comp}
Let $A$ be a set. Then there exist a computable $\mathbf{LTL}$-frame $F=(S, R)$ and a $A$-computable $v: S\times\mathrm{Prop}\rightarrow2$ such that for any valuation $V:S\times \mathrm{Fml}_{\mathbf{LTL}}\rightarrow 2$ on $(S, R)$ extending $v$, $V$ computes $A^{\prime\prime}$.
\end{prop}

\begin{proof}
Let $A$ be a set. For each $e$, let $\Phi_{e}$ denote the $e$-th unary Turing functional. We consider the $\Pi^0_2(A)$-complete set $\mathrm{INF}^{A}$, which is characterized by the following condition:
\begin{align*}
e\in\mathrm{INF}^{A}:\Leftrightarrow\forall s\exists x>s(\Phi^{A}_{e}(x)\downarrow).
\end{align*}

Let $S=\omega$. We define $R=<$, that is, $R$ is the usual order relation on $\omega$.
\vspace{2mm}

Let $\{p_{e}\}_{e\in\omega}$ be an enumeration of $\mathrm{Prop}$. Then we define $v : S\times\mathrm{Prop}\rightarrow2$ by the following: for all $s\in S$ and $e\in\omega$,
\begin{align*}
v(s, p_{e})=1\Leftrightarrow\exists x < s (\Phi^{A\upharpoonright s}_{e}(x)[s] \downarrow \land\, \Phi^{A\upharpoonright (s-1)}_{e}(x)[s-1] \uparrow) \lor (\Phi^{A\upharpoonright s}_{e}(s)[s] \downarrow).
\end{align*}

Fix a $\mathbf{LTL}$-valuation $V:S\times \mathrm{Fml}_{\mathbf{LTL}}\rightarrow 2$ on $(S, R)$ extending $v$. We show that $\mathrm{INF}^{A}\leq_{T}V$. Let $\mathsf{GF} q \equiv \neg (\top \mathsf{U} \neg (\top \mathsf{U} q))$. We show the following: for every $e\in\omega$,
\begin{align*}
V(0,\mathsf{GF} p_{e})=1\Leftrightarrow e\in\mathrm{INF}^{A}.
\end{align*}
$(\Rightarrow)$: We assume that $e\not\in\mathrm{INF}^{A}$. Then, there exists $s\geq0$ such that for all $x>s$, $\Phi^{A}_{e}(x)\uparrow$. Fix such an $s$. Let $k_{0}, \dots, k_{m}$ be the list of all numbers $k < s$ such that $\Phi^{A \upharpoonright t}_{e}(k)[t] \downarrow$ for some stage $t$. Take $l$ large enough such that for each $j \le m$, $\Phi^{A \upharpoonright t}_{e}(k_{j})[t] \downarrow$ at some stage $t < l$. Fix any $h>l+s+1$. 
\vspace{1mm}

We show that $v(h, p_{e})=0$. By assumption, $\Phi^{A}_{e}(h)\uparrow$, that is, $\Phi^{A\upharpoonright h}_{e}(h)[h]\uparrow$. Fix $x<h$. If $s\leq x$, then $\Phi^{A\upharpoonright h}_{e}(x)[h]\uparrow$ by assumption. Assume that $x<s$. If $x=k_{j}$ for some $j$, then $\Phi^{A \upharpoonright (h-1)}_{e}(k_{j})[h-1] \downarrow$. If $x\neq k_{j}$ for any $j$, then $\Phi^{A\upharpoonright h}_{e}(x)[h]\uparrow$. Thus, for all $x<h$, $\Phi^{A\upharpoonright h}_{e}(x)[h] \uparrow \lor\, \Phi^{A\upharpoonright (h-1)}_{e}(x)[h-1] \downarrow$ and $\Phi^{A\upharpoonright h}_{e}(h)[h]\uparrow$, that is, $v(h, p_{e})=0$. Then $V(0,\mathsf{GF} p_{e})=0$.
\vspace{2mm}

$(\Leftarrow)$: We assume that $e\in\mathrm{INF}^{A}$. Fix $s\geq0$. By assumption, there exists $x>s$ such that $\Phi^{A}_{e}(x)\downarrow$. Fix such an $x$. Assume that $\Phi^{A\upharpoonright x}_{e}(x)[x]\uparrow$. Then there exists $t>x$ such that $\Phi^{A\upharpoonright t}_{e}(x)[t] \downarrow \land\, \Phi^{A\upharpoonright (t-1)}_{e}(x)[t-1] \uparrow$. Thus, $V(0,\mathsf{GF} p_{e})=1$.
\vspace{1mm}

Thus, $V$ computes $\mathrm{INF}^{A}$.
\end{proof}

\begin{cor}[$\mathrm{RCA}_0$]
VEL for $\mathbf{LTL}$ implies $\mathrm{ACA}_0$.
\end{cor}

\begin{proof}
By Proposition~\ref{ltla2comp}.
\end{proof}

Next, we consider in which axiom system VEL for $\mathbf{LTL}$ can be proved. The preceding discussion shows that this statement is provable in $\mathrm{ACA}^{+}_0$. For restricted valuations, the following holds.

\begin{prop}[The Restricted Valuation Extension Lemma for $\mathbf{LTL}$, $\mathrm{ACA}_0^{\prime}$]\label{rvelltlacap}
For each $\mathbf{LTL}$-frame $F=(S, R)$, $k\in\mathbb{N}$, and $v:S\times\mathrm{Prop}\rightarrow2$, there exists a $k$-restricted $\mathbf{LTL}$-valuation $V$ on $(S, R)$ such that $V$ is an extension of $v$.
\end{prop}

\begin{proof}
This follows in the same way as Theorem~\ref{pextend2}.
\end{proof}

We further show that VEL for $\mathbf{LTL}$ holds in $\mathrm{ACA}^{\prime}_0$. For this purpose, we use techniques from $\mathbf{LTL}$ model checking. Specifically, we demonstrate this through the following five steps:
\begin{enumerate}
    \item Construct a generalized nondeterministic B\"uchi automaton ($\mathrm{GNBA}$) for the $\mathbf{LTL}$-formula $\varphi$, which accepts a sequence of states if and only if $\varphi$ holds in a sequence of states.
    \item Convert this $\mathrm{GNBA}$ into a nondeterministic B\"uchi automaton ($\mathrm{NBA}$).
    \item Convert this $\mathrm{NBA}$ into a deterministic Rabin automaton ($\mathrm{DRA}$).
    \item The run of the $\mathrm{DRA}$ for any input sequence is unique. Moreover, its acceptance condition can be decided by $(F\oplus v)^{\prime\prime}$. Thus, the truth value of $\varphi$ at that state can be determined using $(F\oplus v)^{\prime\prime}$, and we construct a function deciding the truth value of $\varphi$ using $(F\oplus v)^{\prime\prime}$.
    \item We use $\mathrm{ACA}^{\prime}_0$ to show that the function constructed in $(4)$, which decides the truth value of $\varphi$, is indeed a valuation within the system.
\end{enumerate}

We define the following automata used in the proof. First, we define the automata treated in the proof within $\mathrm{RCA}_0$. The definitions of these automata in $\mathrm{RCA}_0$ are based on~\cite{MR3961566, MR2493187}. Moreover, the proofs of the propositions concerning $(1)$ and $(2)$ below are obtained by formalizing the proofs in~\cite{MR2493187} within $\mathrm{RCA}_0$. 

\begin{dfn}[$\mathrm{RCA}_0$~\cite{MR2493187, MR3961566}]
Let $\Sigma$ be a finite, nonempty set called an \textit{alphabet}. An \textit{infinite word} over $\Sigma$ is a function $\alpha : \mathbb{N}\rightarrow\Sigma$. We write $\alpha\in\Sigma^{\mathbb{N}}$ to mean that $\alpha$ is an infinite word over $\Sigma$.
\end{dfn}

\begin{dfn}[The nondeterministic B\"uchi automaton ($\mathrm{NBA}$), $\mathrm{RCA}_0$~\cite{MR2493187, MR3961566}]
A \textit{nondeterministic B\"uchi automaton} is a tuple $\mathcal{A}=\langle{Q, \Sigma, Q_{0}, \delta, F\rangle}$ where $Q$ is a finite set of \textit{states}, $\Sigma$ is an alphabet, $Q_{0}\subseteq Q$ is a set of \textit{initial states}, $\delta : Q\times\Sigma\times Q\rightarrow 2$ is a \textit{transition relation}, and $F\subseteq Q$ is a set of \textit{accepting states}.

Given an infinite word $\alpha\in\Sigma^{\mathbb{N}}$, we say that $\rho\in Q^{\mathbb{N}}$ is a \textit{run} of $\mathcal{A}$ over $\alpha$ if $\rho(0)\in Q_{0}$ and for every $n\in\mathbb{N}$, $\delta(\rho(n), \alpha(n), \rho(n+1))=1$. A run $\rho$ is accepting if $\rho(n)\in F$ for infinitely many $n\in\mathbb{N}$. A $\mathrm{NBA}$ $\mathcal{A}$ accepts $\alpha$ if there exists an accepting run of $\mathcal{A}$ over $\alpha$ in the above sense. 
\end{dfn}

\begin{dfn}[The generalized nondeterministic B\"uchi automaton ($\mathrm{GNBA}$), $\mathrm{RCA}_0$~\cite{MR2493187, MR3961566}]
A \textit{generalized nondeterministic B\"uchi automaton} is a tuple $\mathcal{A}=\langle{Q, \Sigma, Q_{0}, \delta, \mathcal{F}\rangle}$ where $Q, \Sigma, Q_{0}, \delta$ are defined as for an $\mathrm{NBA}$, and $\mathcal{F}$ is a subset of $2^{Q}$. 

Runs in a $\mathrm{GNBA}$ are defined as for a $\mathrm{NBA}$. A run $\rho$ is accepting if for each $F\in\mathcal{F}$, there exist infinitely many $n\in\mathbb{N}$ such that $\rho(n)\in F$. A $\mathrm{GNBA}$ $\mathcal{A}$ accepts $\alpha$ if there exists an accepting run of $\mathcal{A}$ over $\alpha$ in the above sense. 
\end{dfn}

\begin{dfn}[The deterministic Rabin automaton ($\mathrm{DRA}$), $\mathrm{RCA}_0$~\cite{MR2493187, MR3961566}]
A \textit{deterministic Rabin automaton} is a tuple $\mathcal{A}=\langle{Q, \Sigma, q_{0}, \delta, (E_{i}, F_{i})_{i\leq k}\rangle}$ where $Q$ is a finite set of states, $\Sigma$ is an alphabet, $q_{0}\in Q$ is a starting state, $\delta : Q\times\Sigma\rightarrow Q$ is a (partial) transition function, and $E_{i}, F_{i}\subseteq Q$ for each $i\leq k$.

Runs in a $\mathrm{DRA}$ are defined as for an $\mathrm{NBA}$. A run $\rho$ is accepting if for some $i$, each state in $E_{i}$ appears in $\rho$ only finitely many times and some state in $F_{i}$ appears in $\rho$ infinitely many times. A $\mathrm{DRA}$ $\mathcal{A}$ accepts $\alpha$ if there exists an accepting run of $\mathcal{A}$ over $\alpha$ in the above sense. 
\end{dfn}

Next, we define the objects and notions that appear in the construction of the automaton corresponding to $\varphi\in\mathrm{Fml}_{\mathbf{LTL}}$.

\begin{dfn}[$\mathrm{RCA}_0$~\cite{MR2493187}]
Let $\varphi\in\mathrm{Fml}_{\mathbf{LTL}}$. Let $\mathrm{Sub}(\varphi)$ be the set of all subformulas of $\varphi$. We define $\sim\, : \mathrm{Fml}_{\mathbf{LTL}}\rightarrow\mathrm{Fml}_{\mathbf{LTL}}$ as follows: for all $\psi\in\mathrm{Fml}_{\mathbf{LTL}}$,
\begin{align*}
\sim\psi = 
\begin{cases}
\rho & \text{if } \exists\rho \, (\psi \equiv \neg\rho), \\
\neg\psi & \text{otherwise.}
\end{cases}
\end{align*}
Then, we let $\overline{\mathrm{Sub}}(\varphi)=\mathrm{Sub}(\varphi)\cup\{\sim\psi\mid\psi\in\mathrm{Sub}(\varphi)\}$.
\end{dfn}

\begin{dfn}[$\mathrm{RCA}_0$~\cite{MR2493187}]
Let $\varphi\in\mathrm{Fml}_{\mathbf{LTL}}$ and $B\subseteq\overline{\mathrm{Sub}}(\varphi)$.
\begin{itemize}
    \item $B$ is \textit{consistent} with respect to propositional logic if for all $\varphi_{0}\land\varphi_{1}, \psi\in\overline{\mathrm{Sub}}(\varphi)$,
    \begin{itemize}
        \item $\varphi_{0}\land\varphi_{1}\in B\Leftrightarrow\varphi_{0}\in B$ and $\varphi_{1}\in B$,
        \item $\psi\in B\Rightarrow\neg\psi\not\in B$,
        \item $\top\in\overline{\mathrm{Sub}}(\varphi)\Rightarrow\top\in B$.
    \end{itemize}
    \item $B$ is \textit{locally consistent} with respect to the until operator if for all $\varphi_{0}\mathsf{U}\varphi_{1}\in\overline{\mathrm{Sub}}(\varphi)$,
    \begin{itemize}
        \item $\varphi_{1}\in B\Rightarrow\varphi_{0}\mathsf{U}\varphi_{1}\in B$,
        \item $\varphi_{0}\mathsf{U}\varphi_{1}\in B$ and $\varphi_{1}\not\in B\Rightarrow\varphi_{0}\in B$.
    \end{itemize}
    \item $B$ is \textit{maximal} if for all $\psi\in\overline{\mathrm{Sub}}(\varphi)$, $\psi\not\in B\Rightarrow\neg\psi\in B$.
    \item $B$ is \textit{elementary} if it is consistent with respect to propositional logic, locally consistent with respect to the until operator, and maximal.
\end{itemize}
\end{dfn}

We now show that the construction in $(1)$ works.

\begin{prop}[$\mathrm{RCA}_0$~\cite{MR2493187}]\label{gnbatov}
Let $\varphi \in \mathrm{Fml}_{\mathbf{LTL}}$ with $\mathrm{deg}(\varphi)=k$, and let $\Sigma_{\varphi}$ be the set of all subsets of $\mathrm{Prop} \cap \overline{\mathrm{Sub}}(\varphi)$. 
Then, there exists a $\mathrm{GNBA}$ $\mathcal{G}_{\varphi}$ over the alphabet $\Sigma_{\varphi}$ satisfying the following property: 
For any $\mathbf{LTL}$-frame $F=(S, R)$, any $v : S \times \mathrm{Prop} \rightarrow 2$, and any $s_{0} \in S$, 
there exists an infinite word $\alpha_0 \in \Sigma_{\varphi}^{\mathbb{N}}$ such that for the $k$-restricted $\mathbf{LTL}$-valuation $V$ on $F$ extending $v$,
\begin{align*}
V(s_{0}, \varphi) = 1 \iff \mathcal{G}_{\varphi} \text{ accepts } \alpha_0.
\end{align*}
\end{prop}
\begin{proof}
Let $\varphi\in\mathrm{Fml}_{\mathbf{LTL}}$ with $\mathrm{deg}(\varphi)=k$, and let $\Sigma_{\varphi}$ be the set of all subsets of $\mathrm{Prop}\cap\overline{\mathrm{Sub}}(\varphi)$. Then we define $\mathcal{G}_{\varphi}=\langle{Q, \Sigma_{\varphi}, Q_{0}, \delta, \mathcal{F}\rangle}$ as follows:
\begin{itemize}
    \item $Q$ is the set of all elementary sets of formulas $B\subseteq\overline{\mathrm{Sub}}(\varphi)$,
    \item $Q_{0}=\{B\in Q\mid\varphi\in B\}$,
    \item $\mathcal{F}=\{F_{\varphi_{0}\mathsf{U}\varphi_{1}}\mid\varphi_{0}\mathsf{U}\varphi_{1}\in\overline{\mathrm{Sub}}(\varphi)\}$ where $F_{\varphi_{0}\mathsf{U}\varphi_{1}}=\{B\in Q\mid\varphi_{0}\mathsf{U}\varphi_{1}\not\in B\text{ or }\varphi_{1}\in B\}$.
\end{itemize}

We define the transition relation $\delta\, : Q\times\Sigma_{\varphi}\times Q\rightarrow 2$ as follows: for all $B, B^{\prime}\in Q$ and $A\in\Sigma_{\varphi}$,
\begin{align*}
\delta(B, A, B^{\prime})=1\Leftrightarrow&\,A=\mathrm{Prop}\cap B\\&\land
\forall\,\mathsf{X}\psi\in\overline{\mathrm{Sub}}(\varphi)(\mathsf{X}\psi\in B\leftrightarrow\psi\in B^{\prime})
\\&\land\forall\,\varphi_{0}\mathsf{U}\varphi_{1}\in\overline{\mathrm{Sub}}(\varphi)\,\Bigl(\varphi_{0}\mathsf{U}\varphi_{1}\in B\leftrightarrow(\varphi_{1}\in B\lor(\varphi_{0}\in B\land\varphi_{0}\mathsf{U}\varphi_{1}\in B^{\prime}))\Bigr).
\end{align*}

Fix a $\mathbf{LTL}$-frame $F=(S, R)$, $v : S\times\mathrm{Prop}\rightarrow 2$, and $s_{0}\in S$. Let $s : \mathbb{N}\rightarrow S$ be the unique path starting from $s_{0}$ by $R$. For each $i$, we let $A_{i}=\{p\in\mathrm{Prop}\cap\overline{\mathrm{Sub}}(\varphi)\mid v(s_{i}, p)=1\}$. Then we have $\alpha_0=A_{0}A_{1}\dots$. Next, we show that $V(s_{0}, \varphi)=1\Leftrightarrow \mathcal{G}_{\varphi}\text{ accepts $\alpha_0$}$. Fix a $k$-restricted $\mathbf{LTL}$-valuation $V$ on $F$ extending $v$.
\vspace{3mm}

$(\Rightarrow)$: Assume that $V(s_{0}, \varphi)=1$. For each $i$, we let $B_{i}=\{\psi\in\overline{\mathrm{Sub}}(\varphi)\mid V(s_{i}, \psi)=1\}$. We show that $\beta=B_{0}B_{1}\dots$ is an accepting run for $\alpha_0$. 
\vspace{2mm}

First, we show that $\beta$ is a run of $\mathcal{G}_{\varphi}$. Clearly, for each $i$, $A_{i}=\mathrm{Prop}\cap B_{i}$. For all $\mathsf{X}\psi\in\overline{\mathrm{Sub}}(\varphi)$ and $i$, 
\begin{align*}
\mathsf{X}\psi\in B_{i}&\Leftrightarrow V(s_{i}, \mathsf{X}\psi)=1\\
&\Leftrightarrow V(s_{i+1}, \psi)=1\\
&\Leftrightarrow \psi\in B_{i+1}.
\end{align*}
Similarly, one shows that for all $\varphi_{0}\mathsf{U}\varphi_{1}\in\overline{\mathrm{Sub}}(\varphi)$ and $i$, $\varphi_{0}\mathsf{U}\varphi_{1}\in B_{i}$ if and only if $\varphi_{1}\in B_{i}$ or ($\varphi_{0}\in B_i$ and $\varphi_{0}\mathsf{U}\varphi_{1}\in B_{i+1}$). Thus, for all $i$, $\delta(B_{i}, A_{i}, B_{i+1})=1$. 
\vspace{1mm}

Next, we show that $\beta$ is accepting. To derive a contradiction, we assume that there exist $\varphi_{0, j}\mathsf{U}\varphi_{1, j}\in\overline{\mathrm{Sub}}(\varphi)$ and $m$ such that for all $l>m$, $B_{l}\not\in F_{\varphi_{0, j}\mathsf{U}\varphi_{1, j}}$. By assumption, $B_{m+1}\not\in F_{\varphi_{0, j}\mathsf{U}\varphi_{1, j}}$, that is, $\varphi_{0, j}\mathsf{U}\varphi_{1, j}\in B_{m+1}$ and $\varphi_{1}\not\in B_{m+1}$. Then $V(s_{m+1}, \varphi_{0, j}\mathsf{U}\varphi_{1, j})=1$ and $V(s_{m}, \varphi_{1, j})=0$. Thus, for some $k>m+1$, $V(s_{k}, \varphi_{1, j})=1$. By the definition of $B_{k}$, $\varphi_{1, j}\in B_{k}$, that is, $B_{k}\in F_{\varphi_{0, j}\mathsf{U}\varphi_{1, j}}$. It is a contradiction. Thus, $\beta$ is an accepting run for $\alpha_0$. 
\vspace{3mm}

$(\Leftarrow)$: Assume that $\mathcal{G}_{\varphi}$ accepts $\alpha_0$. Then there exists an accepting run $\beta=B_{0}B_{1}\dots$ for $\alpha_0$ in $\mathcal{G}_{\varphi}$. Then it follows that for all $i$, $A_{i}=\mathrm{Prop}\cap B_{i}$. 
\vspace{1mm}

For the above $\beta$, we show the following lemmas.
\begin{lem}[$\mathrm{RCA}_0$]\label{ltllemma1}
Let $\varphi_{0}\mathsf{U}\varphi_{1}\in\overline{\mathrm{Sub}}(\varphi)$ and fix $m<j$. Suppose that for all $m\leq i<j$, $\varphi_{0}\in B_{i}$ and $\varphi_{1}\in B_{j}$. Then $\varphi_{0}\mathsf{U}\varphi_{1}\in B_{m}$. 
\end{lem}
\begin{lemproof}
Let $\varphi_{0}\mathsf{U}\varphi_{1}\in\overline{\mathrm{Sub}}(\varphi)$ and fix $m<j$. Suppose that for all $m\leq i<j$, $\varphi_{0}\in B_{i}$ and $\varphi_{1}\in B_{j}$. Let $l=j-m$. Then we show the following statement by $\Sigma^0_0$-induction on $k\leq l$:
\begin{align*}
\varphi_{0}\mathsf{U}\varphi_{1}\in B_{j-k}. 
\end{align*}

\textit{Base case}: By assumption, $\varphi_{1}\in B_{j}$. Since $B_{j}$ is elementary, $\varphi_{0}\mathsf{U}\varphi_{1}\in B_{j}$.
\vspace{2mm}

\textit{Induction step}: Let $k<l$. Assume that $\varphi_{0}\mathsf{U}\varphi_{1}\in B_{j-k}$. By the assumption, $\varphi_{0}\in B_{j-(k+1)}$. Then, by definition of the transition relation, $\varphi_{0}\mathsf{U}\varphi_{1}\in B_{j-(k+1)}$. 
\vspace{2mm}

By induction, we complete the proof; that is, $\varphi_{0}\mathsf{U}\varphi_{1}\in B_{m}$. 
\end{lemproof}

\begin{lem}[$\mathrm{RCA}_0$]\label{ltllemma2}
Let $\varphi_{0}\mathsf{U}\varphi_{1}\in\overline{\mathrm{Sub}}(\varphi)$ and fix $m<j$. Suppose that for all $m\leq i\leq j$, $\varphi_{1}\not\in B_{i}$ and $\varphi_{0}\mathsf{U}\varphi_{1}\in B_{m}$. Then for all $m\leq i\leq j$, $\varphi_{0}\in B_{i}$ and $\varphi_{0}\mathsf{U}\varphi_{1}\in B_{i}$. 
\end{lem}
\begin{lemproof}
Let $\varphi_{0}\mathsf{U}\varphi_{1}\in\overline{\mathrm{Sub}}(\varphi)$ and fix $m<j$. Suppose that for all $m\leq i\leq j$, $\varphi_{1}\not\in B_{i}$ and $\varphi_{0}\mathsf{U}\varphi_{1}\in B_{m}$. Let $l=j-m$. Then we show the following statement by $\Sigma^0_0$-induction on $k\leq l$:
\begin{align*}
\varphi_{0}\in B_{m+k}\text{ and }\varphi_{0}\mathsf{U}\varphi_{1}\in B_{m+k}. 
\end{align*}

\textit{Base case}: By assumption, $\varphi_{1}\not\in B_{m}$ and $\varphi_{0}\mathsf{U}\varphi_{1}\in B_{m}$. Since $B_{m}$ is elementary, $\varphi_{0}\in B_{m}$.
\vspace{2mm}

\textit{Induction step}: Let $k<l$. Assume that $\varphi_{0}\in B_{m+k}\text{ and }\varphi_{0}\mathsf{U}\varphi_{1}\in B_{m+k}$. By assumption, $\varphi_{1}\not\in B_{m+k}$ and $\varphi_{1}\not\in B_{m+(k+1)}$. Then, by the definition of the transition relation, $\varphi_{0}\mathsf{U}\varphi_{1}\in B_{m+(k+1)}$. Since $B_{m+(k+1)}$ is elementary, $\varphi_{0}\in B_{m+(k+1)}$.
\vspace{2mm}

By induction, we complete the proof. 
\end{lemproof}

Next, we show that the following statement by $\Sigma^0_1$-induction on the complexity of the formula $\psi\in\overline{\mathrm{Sub}}(\varphi)$:
\begin{align*}
\forall i(\psi\in B_{i}\leftrightarrow V(s_{i}, \psi)=1).
\end{align*}

\textit{Base case}: If $\psi\equiv p\in\mathrm{Prop}$, then $\psi\in B_{0}\Leftrightarrow V(s_{0}, \psi)=1$ by definition.
\vspace{2mm}

\textit{Induction step}: It is enough to show that it holds for the case of $\psi\equiv\varphi_{0}\mathsf{U}\varphi_{1}$. Fix $\varphi_{0}, \varphi_{1}\in\overline{\mathrm{Sub}}(\varphi)$. Assume that for each $k<2$, $\forall j(\varphi_{k}\in B_{j}\leftrightarrow V(s_{j}, \varphi_{k})=1)$ holds. Fix $m$.
\begin{itemize}
    \item Assume that $V(s_{m}, \varphi_{0}\mathsf{U}\varphi_{1})=1$. Then, there exists $j\geq m$ such that for all $m\leq i< j$, $V(s_{i}, \varphi_{0})=1$ and $V(s_{j}, \varphi_{1})=1$. By assumption, for all $m\leq i< j$, $\varphi_{0}\in B_{i}$ and $\varphi_{1}\in B_{j}$. Then $\varphi_{0}\mathsf{U}\varphi_{1}\in B_{j}$. By Lemma~\ref{ltllemma1}, $\varphi_{0}\mathsf{U}\varphi_{1}\in B_{m}$.
    \item Assume that $\varphi_{0}\mathsf{U}\varphi_{1}\in B_{m}$. Since $B_{m}$ is elementary, $\varphi_{0}\in B_{m}$ or $\varphi_{1}\in B_{m}$. If $\varphi_{1}\in B_{m}$, then, by assumption, $V(s_{m}, \varphi_{1})=1$. Thus, $V(s_{m},\varphi_{0}\mathsf{U}\varphi_{1})=1$. We assume that $\varphi_{0}\in B_{m}$. To derive a contradiction, assume that for all $j\geq m$, $\varphi_{1}\not\in B_{j}$. Then, by Lemma~\ref{ltllemma2}, for all $j\geq m$, $\varphi_{0}\in B_{j}$ and $\varphi_{0}\mathsf{U}\varphi_{1}\in B_{j}$. Thus, for infinitely many $j\geq m$, $B_{j}\in F_{\varphi_{0}\mathsf{U}\varphi_{1}}$. However, for all $j$, we have $\varphi_{1}\not\in B_{j}$ and $\varphi_{0}\mathsf{U}\varphi_{1}\in B_{j}$, that is, $B_{j}\not\in F_{\varphi_{0}\mathsf{U}\varphi_{1}}$. It is a contradiction. Thus, for some $j\geq m$, $\varphi_{1}\in B_{j}$. Let $j$ be the smallest index such that $\varphi_{1}\in B_{j}$. We may consider $j>m$. By Lemma~\ref{ltllemma2}, for all $m\leq i<j$, $\varphi_{0}\in B_{i}\text{ and }\varphi_{0}\mathsf{U}\varphi_{1}\in B_{i}$. By assumption, for all $m\leq i<j$, $V(s_{i}, \varphi_{0})=1$ and $V(s_{j}, \varphi_{1})=1$. Thus, we have $V(s_{m}, \varphi_{0}\mathsf{U}\varphi_{1})=1$.
\end{itemize}
By induction, we complete the proof. 
\end{proof}

Next, we show that the construction in $(2)$ works.

\begin{prop}[$\mathrm{RCA}_0$~\cite{MR2493187}]\label{gnbatonba}
For each $\mathrm{GNBA}$ $\mathcal{G}$ over an alphabet $\Sigma$, there exists a $\mathrm{NBA}$ $\mathcal{B}$ over the same alphabet such that for every infinite word $\alpha\in\Sigma^{\mathbb{N}}$,
\begin{align*}
\mathcal{G}\text{ accepts $\alpha$}\Leftrightarrow\mathcal{B}\text{ accepts $\alpha$}.
\end{align*}
\end{prop}
\begin{proof}
Let $\mathcal{G}=\langle{Q, \Sigma, Q_{0}, \delta, \mathcal{F}\rangle}$ be a $\mathrm{GNBA}$. In fact, $\mathcal{G}=\langle{Q, \Sigma, Q_{0}, \delta, \mathcal{F}\rangle}$ is equivalent to a $\mathcal{G}^{\prime}=\langle{Q, \Sigma, Q_{0}, \delta, \mathcal{F}\cup\{Q\}\rangle}$. Thus, we may assume that $\mathcal{F}\neq\emptyset$. Let $\mathcal{F}=\{F_{0},\dots, F_{k}\}$ where $k\geq0$. Then we define $\mathcal{B}=\langle{Q^{\prime}, \Sigma, Q^{\prime}_{0}, \delta^{\prime}, F^{\prime}\rangle}$ as follows:
\begin{itemize}
    \item $Q^{\prime}=Q\times\{0,\dots,k\}$,
    \item $Q^{\prime}_{0}=Q_{0}\times\{0\}$,
    \item $F^{\prime}=F_{0}\times\{0\}$.
\end{itemize}

We define the transition relation $\delta^{\prime}\, : Q^{\prime}\times\Sigma\times Q^{\prime}\rightarrow 2$ as follows: for all $\langle{q, i\rangle}, \langle{q^{\prime}, i^{\prime}\rangle}\in Q^{\prime}$ and $A\in\Sigma$,
\begin{align*}
\delta^{\prime}(\langle{q, i\rangle}, A, \langle{q^{\prime}, i^{\prime}\rangle})=1\Leftrightarrow&\,\bigl(i^{\prime}=i\land\delta(q, A, q^{\prime})=1\land q\not\in F_{i}\bigr)\\
&\lor\bigl(i^{\prime}=i+1\leq k\land\delta(q, A, q^{\prime})=1\land q\in F_{i}\bigr)\\
&\lor\bigl(i=k\land i^{\prime}=0\land\delta(q, A, q^{\prime})=1\land q\in F_{k}\bigr).
\end{align*}

Fix an infinite word $\alpha\in\Sigma^{\mathbb{N}}$. We show that $\mathcal{G}\text{ accepts $\alpha$}\Leftrightarrow\mathcal{B}\text{ accepts $\alpha$}$. We may assume that $k=1$.
\vspace{2mm}

$(\Leftarrow)$: Assume that $\mathcal{B}$ accepts $\alpha$. Then there exists an accepting run $\beta=\langle{q_{0}, i_{0}\rangle}\langle{q_{1}, i_{1}\rangle}\dots$ for $\alpha$ in $\mathcal{B}$. We show that $\beta^{\prime}=q_{0}q_{1}\dots$ is an accepting run in $\mathcal{G}$. 
\vspace{1mm}

Since $\beta$ is an accepting run in $\mathcal{B}$, there exists an infinite increasing sequence $k_{0}, k_{1}, \dots$ such that $\langle{q_{k_{m}}, i_{k_{m}}\rangle}=\langle{q_{k_{m}}, 0\rangle}\in F_{0}\times\{0\}$ for all $m$. Fix any $m$ and let $s_{0}\coloneqq k_{m}$. We show the following statement by $\Sigma^0_1$-induction on $n\leq k$.
\begin{align*}
\exists s_1\dots\exists s_n(s_{0}< s_{1}<\dots< s_{n}\land\forall j\leq n(\langle{q_{s_{j}}, i_{s_{j}}\rangle}\in F_{j}\times\{j\})).
\end{align*}

\textit{Base case}: It is obvious.
\vspace{2mm}

\textit{Induction step}: Fix $n<k$, and we assume that $s_{1},\dots,s_{n}$ satisfy the above statement. Then $\langle{q_{(s_{n}+1)}, i_{(s_{n}+1)}\rangle}\in Q\times\{n+1\}$ by the transition relation. Let $s_{(n+1)}$ be the least number $l>s_{n}$ such that $i_{(l+1)}\neq n+1$. Then $i_{s_{(n+1)}}=n+1$ and $q_{s_{(n+1)}}\in F_{i_{s_{(n+1)}}}=F_{n+1}$ by the transition relation. Thus, $\langle{q_{s_{(n+1)}}, i_{s_{(n+1)}}\rangle}\in F_{n+1}\times\{n+1\}$.
\vspace{2mm}

By induction, for all $n \le k$, there exist $s_{0} < s_{1} < \dots < s_{n}$ such that $\langle q_{s_{j}}, i_{s_{j}} \rangle \in F_{j} \times \{j\}$ for all $j \le n$. Since $m$ was arbitrary, it follows that for each $j \le k$, $\langle q_{m}, i_{m} \rangle \in F_{j} \times \{j\}$ for infinitely many $m$, that is, $q_{m} \in F_{j}$ for infinitely many $m$. Thus, $\beta'$ is accepting in $\mathcal{G}$.
\vspace{2mm}

$(\Rightarrow)$: Assume that $\mathcal{G}$ accepts $\alpha$. Then there exists an accepting run $\beta=q_{0}q_{1}\dots$ for $\alpha$ in $\mathcal{G}$. Let $i_{0}=0$. We define the sequence $\{i_{n}\}_{n}$ as follows: assume that $i_{n}$ is defined. Then,
\begin{align*}
i_{(n+1)} \coloneqq 
\begin{cases}
i_{n} & \text{if } q_{n}\not\in F_{i_{n}}, \\
i_{n}+1 & \text{if } q_{n}\in F_{i_{n}}\land i_{n}<k, \\
0 & \text{otherwise.}
\end{cases}
\end{align*}

We show that $\beta^{\prime}=\langle{q_{0}, i_{0}\rangle}\langle{q_{1}, i_{1}\rangle}\dots$ is an accepting run in $\mathcal{B}$. Let $s_{0}$ be the least number $l\geq0$ such that $i_{(l+1)}\neq0$. Then $i_{s_{0}}=0$ and $q_{s_{0}}\in F_{i_{s_{0}}}=F_{0}$ by definition of $i_{n}$. Thus, $\langle{q_{s_{0}}, i_{s_{0}}\rangle}\in F_{0}\times\{0\}$. Similarly to the argument above, for each $n \le k$, there exist $s_{0} < s_{1} < \dots < s_{n}$ such that $\langle q_{s_{j}}, i_{s_{j}} \rangle \in F_{j} \times \{j\}$ for all $j \le n$. Furthermore, we may show that for any $m$ such that $\langle q_{m}, i_{m} \rangle \in F_{k} \times \{k\}$, there exists some $s > m$ such that $\langle q_{s}, i_{s} \rangle \in F_{0} \times \{0\}$. Thus, it follows that $\langle q_{m}, i_{m} \rangle \in F_{0} \times \{0\}$ for infinitely many $m$. Thus, $\beta^{\prime}$ is accepting in $\mathcal{B}$.
\end{proof}

The fact that the construction in $(3)$ works follows from the following proposition.

\begin{prop}[$\mathrm{RCA}_0+\Sigma^0_2$-induction~\cite{MR3961566}]\label{gnbatodra}
The existence of an algorithm which, given a $\mathrm{NBA}$ $\mathcal{B}$ over an alphabet $\Sigma$, outputs an equivalent $\mathrm{DRA}$ $\mathcal{A}$ over the same alphabet such that for every infinite word $\alpha\in\Sigma^{\mathbb{N}}$ we have
\begin{align*}
\mathcal{B}\text{ accepts $\alpha$}\Leftrightarrow\mathcal{A}\text{ accepts $\alpha$}.
\end{align*}
\end{prop}

From the above, we obtain the following.

\begin{thm}[The Valuation Extension Lemma for $\mathbf{LTL}$, $\mathrm{ACA}_0^{\prime}$]\label{velltlacapf}
For each $\mathbf{LTL}$-frame $F=(S, R)$ and $v:S\times\mathrm{Prop}\rightarrow2$, there exists a $\mathbf{LTL}$-valuation $V$ on $(S, R)$ such that $V$ is an extension of $v$.
\end{thm}

\begin{proof}
Let $F=(S, R)$ be an $\mathbf{LTL}$-frame and $v:S\times\mathrm{Prop}\rightarrow2$. Fix a formula $\varphi$. Let $\Sigma_{\varphi}$ be the set of all subsets of $\mathrm{Prop} \cap \overline{\mathrm{Sub}}(\varphi)$. Let $\mathcal{G}_{\varphi}$ be the $\mathrm{GNBA}$ obtained from Proposition~\ref{gnbatov}. By Proposition~\ref{gnbatonba}, we have an $\mathrm{NBA}$ $\mathcal{B}_{\varphi}$ equivalent to $\mathcal{G}_{\varphi}$. Furthermore, by Proposition~\ref{gnbatodra}, we obtain a $\mathrm{DRA}$ $\mathcal{A}_{\varphi}$ equivalent to $\mathcal{B}_{\varphi}$. Fix $s_{0}\in S$. Let $\alpha_0\in\Sigma^{\mathbb{N}}_{\varphi}$ be the infinite word defined in Proposition~\ref{gnbatov}. Under this setup, we show that whether $\mathcal{A}_{\varphi}$ accepts $\alpha_0$ can be decided by $(F\oplus v)^{\prime\prime}$.
\vspace{2mm}

Let $\mathcal{A}_{\varphi}=\langle{Q, \Sigma_{\varphi}, q_{0}, \delta, (E_{i}, F_{i})_{i\leq k}\rangle}$. The unique path $\alpha_0=A_{0}A_{1}\dots$ of $s_{0}$ is $(F\oplus v)$-computable. First, we query the oracle $(F\oplus v)^{\prime}$ to determine whether the following condition holds:
\begin{align*}
\forall n\exists q_1\dots\exists q_n\in Q(\forall i\leq n\,\delta(q_i, A_i, q_{i+1})=1).
\end{align*}
If the answer is no, then $\mathcal{A}_{\varphi}$ has no run on $\alpha_0$, and thus $\mathcal{A}_{\varphi}$ does not accept $\alpha_0$. If the answer is yes, we can construct the infinite run $\{q_{i}\}_{i}$ on $\alpha_0$ effectively by using $F\oplus v$ as an oracle. Note that this run is uniquely determined since $\mathcal{A}_{\varphi}$ is a DRA. Next, we query the oracle $(F\oplus v)^{\prime\prime}$ to determine whether:
\begin{align*}
\exists j\leq k\Bigl(\bigl(\exists m(\forall n>m\, q_n\not\in E_j)\bigr)\land\bigl(\forall m^{\prime}(\exists n^{\prime}>m^{\prime}\,q_{n^{\prime}}\in F_j)\bigr)\Bigr).
\end{align*}
If the answer is no, $\{q_{i}\}_{i}$ is not an accepting run, implying that $\mathcal{A}_{\varphi}$ does not accept $\alpha_0$. If the answer is yes, $\{q_{i}\}_{i}$ is indeed an accepting run, and we conclude that $\mathcal{A}_{\varphi}$ accepts $\alpha_0$.

Thus, we define the function $f : S\times\mathrm{Fml}_{\mathbf{LTL}}\rightarrow 2$ as follows: for all $s_0\in S$ and $\varphi\in\mathrm{Fml}_{\mathbf{LTL}}$,
\begin{align*}
f(s_0, \varphi)=1:\Leftrightarrow\mathcal{A}_{\varphi}\text{ accepts }\alpha_0.
\end{align*}

This $f$ is $(F\oplus v)^{\prime\prime}$-computable, thus $f$ exists as a set. We show that $f$ is a $\mathbf{LTL}$-valuation on $F$. Fix $s_0\in S$ and $\varphi\in\mathrm{Fml}_{\mathbf{LTL}}$ with $\mathrm{deg}(\varphi)=k$. Let $\alpha_0=A_0A_1\dots\in\Sigma^{\mathbb{N}}_{\varphi}$ be the infinite word defined in Proposition~\ref{gnbatov}. For all $p\in\mathrm{Prop}$, $f(s_0, p)=1\Leftrightarrow v(s_{0}, p)=1$ by definition of $f$, Proposition~\ref{gnbatov}, Proposition~\ref{gnbatonba}, and Proposition~\ref{gnbatodra}. Thus, $f$ is an extension of $v$.
\vspace{1mm}

Next, we show that the following holds: for every $\varphi\in\mathrm{Fml}_{\mathbf{LTL}}$, 
\begin{align*}
f(s_0, \varphi)=1\Leftrightarrow f(s_0, \neg\varphi)=0.
\end{align*}

By Proposition~\ref{rvelltlacap} and $\mathrm{ACA}^{\prime}_0$, we obtain a $k$-restricted $\mathbf{LTL}$-valuation $V$. By Proposition~\ref{gnbatov}, Proposition~\ref{gnbatonba}, and Proposition~\ref{gnbatodra}, 
\begin{align*}
V(s_0, \varphi)=1&\Leftrightarrow\mathcal{A}_{\varphi}\text{ accepts }\alpha_0,\\
V(s_0, \neg\varphi)=0&\Leftrightarrow\mathcal{A}_{\neg\varphi}\text{ does not accept }\alpha_0.
\end{align*}
Since $V$ is a $k$-restricted $\mathbf{LTL}$-valuation, we obtain the following.
\begin{align*}
\mathcal{A}_{\varphi}\text{ accepts }\alpha_0\Leftrightarrow\mathcal{A}_{\neg\varphi}\text{ does not accept }\alpha_0.
\end{align*}

For each $s_i$ on the unique path starting from $s_0$, let $\alpha_i$ be the corresponding infinite word. Noting that $\alpha_i=A_iA_{i+1}\dots$, we can also verify the other properties of a $\mathbf{LTL}$-valuation.
\end{proof}

\section{Conclusion and Future Research}
In this paper, we have shown not only that frame definability for Geach axioms and for $\mathbf{GL}$, both of which are central objects of study in modal logic, is equivalent to $\mathrm{ACA}^{+}_0$ over $\mathrm{RCA}_0$, but also that the main source of difficulty in frame definability lies in the existence of full valuations. More precisely, we have shown that the equivalence between VEL and $\mathrm{ACA}^{+}_0$ captures the essential reverse-mathematical content of the problem. Furthermore, we have proved that, in Kripke semantics, the computational complexity of valuations is determined more by the computable frame itself than by a computable assignment of propositional variables to states. Thus, in general, $\mathrm{RCA}_0$ cannot prove that a frame has a valuation on it, even when the valuation is intended to be trivial. 
\vspace{2mm}

On the other hand, the reverse-mathematical analysis of VEL for $\mathbf{CTL}$ and $\mathbf{LTL}$ is not yet complete. VEL for $\mathbf{CTL}$ lies between $\Sigma^1_2$-$\mathrm{DC}_0$ and $\Pi^1_1$-$\mathrm{CA}_0$, but its exact position is still unknown. Note that the statement of VEL for $\mathbf{CTL}$ can be written as a $\Pi^1_3$ sentence. If VEL for $\mathbf{CTL}$ is provable in $\Sigma^1_2\text{-$\mathrm{AC}_0$}$ without $\Sigma^1_2$-induction, then, by the $\Pi^1_3$-conservativity of $\Sigma^1_2\text{-$\mathrm{AC}_0$}$ over $\Pi^1_1\text{-$\mathrm{CA}_0$}$, VEL for $\mathbf{CTL}$ is equivalent to $\Pi^1_1\text{-$\mathrm{CA}_0$}$ over $\mathrm{RCA}_0$. 

\begin{que}
Does $\Sigma^1_2\text{-$\mathrm{AC}_0$}$ imply (Restricted) VEL for $\mathbf{CTL}$?
\end{que}

If the above question has an affirmative answer, then both VEL for $\mathbf{CTL}$ and Restricted VEL for $\mathbf{CTL}$ are equivalent to $\Pi^1_1\text{-$\mathrm{CA}_0$}$ over $\mathrm{RCA}_0$.
\vspace{2mm}

As for VEL for $\mathbf{LTL}$, its reverse-mathematical strength also lies between $\mathrm{ACA}^{\prime}_0$ and $\mathrm{ACA}_0$, but its exact position remains open. Given a computable $\mathbf{LTL}$-frame $F$ and a computable assignment $v$ of propositional variables to states, Theorem~\ref{velltlacapf} shows that the corresponding $0,1$-valued function on $\mathbf{LTL}$-formulas and states can be constructed as a $\emptyset^{\prime\prime}$-computable function. On the other hand, if there is an $\mathbf{LTL}$-valuation $V$ on $F$ extending $v$, then no matter how computationally complicated $V$ may appear, the uniqueness of the valuation extending $v$ implies that $V$ must coincide with $f$. Hence this $V$ cannot compute $\emptyset^{\prime\prime\prime}$. Thus, from the viewpoint of computability theory, it is natural to conjecture that VEL for $\mathbf{LTL}$ is equivalent to $\mathrm{ACA}_0$. However, in the standard $\mathbf{LTL}$ model-checking argument, the proof that the truth of an $\mathbf{LTL}$-formula $\varphi$ can be decided by a GNBA relies on the existence of a valuation on the $\mathbf{LTL}$-frame. The current task is to avoid this circularity and prove in $\mathrm{ACA}_0$ that the truth of $\varphi$ can be decided by a GNBA.

\begin{que}
Let $\varphi \in \mathrm{Fml}_{\mathbf{LTL}}$, let $F=(S, R)$ be a $\mathbf{LTL}$-frame, let $v : S \times \mathrm{Prop} \rightarrow 2$, and let $s_{0} \in S$. Let the GNBA $\mathcal{G}_{\varphi}$ and the infinite word $\alpha_0$ be as in Proposition~\ref{gnbatov}. Can we prove in $\mathrm{ACA}_0$ that the assertion ``$\mathcal{G}_{\varphi} \text{ accepts } \alpha_0$''  satisfies the conditions for an $\mathbf{LTL}$-valuation?
\end{que}

\bibliography{refmo}
\begin{comment}
\newpage

\renewcommand{\thesection}{\Alph{section}}
\setcounter{section}{1}

\section*{Appendix A: Some proofs}

\subsection{Proofs of Lemmas in Section~3}

%\subsection{Proofs of Lemmas in Section~4}

%\subsection{Proof of Lemma~\ref{frightk}}

\stepcounter{section}
\end{comment}
\end{document}